\documentclass{article}
\usepackage[english]{babel} 		% Packages add functionality and style conventions to your documents. Don't edit this section
\usepackage{fullpage}				% Eliminates wasted space
\usepackage[utf8]{inputenc}			% Necessary for character encoding
\usepackage{amsmath, amssymb,amsthm}% Required math packages
\usepackage{graphicx}				% For handling graphics
\usepackage{hyperref}				% For clickable links in the final PDF
\usepackage{xcolor}
\usepackage{multirow}
\usepackage{tikz}
\usepackage{ytableau}
\usepackage{mathrsfs}
\usepackage{thmtools}
\usepackage{thm-restate}
\usepackage{varwidth}
\usepackage{comment}
\usepackage{mathtools}
\usepackage[mathscr]{euscript}
\usepackage{bbm}
\usepackage{multicol}
\usepackage{enumerate}

\usepackage[
    backend=bibtex,
    maxnames=5,
    style=alphabetic
  ]{biblatex}
\theoremstyle{definition}
\newtheorem{thm}{Theorem}
\newtheorem{lem}[thm]{Lemma}
\newtheorem*{lem*}{Lemma}
\newtheorem*{thm*}{Theorem}
\newtheorem{prop}[thm]{Proposition}

\newtheorem{cor}[thm]{Corollary}

\newtheorem{defn}[thm]{Definition}
\newtheorem*{remark*}{Remark}
\newtheorem{remark}{Remark}
\newtheorem*{example}{Example}
\newtheorem{conjecture}{Conjecture}
\newtheorem{cor/defn}[thm]{Corollary/Definition}
\newcommand{\y}{{ \mathscr{Y} }}

\DeclareMathOperator{\sP}{\mathscr{P} }
\DeclareMathOperator{\WP}{WP}
\DeclareMathOperator{\area}{area}
\DeclareMathOperator{\dinv}{dinv}
\DeclareMathOperator{\mathleft}{left}
\DeclareMathOperator{\mathright}{right}
\DeclareMathOperator{\arm}{arm}
\DeclareMathOperator{\leg}{leg}
\DeclareMathOperator{\Des}{Des}
\DeclareMathOperator{\maj}{maj}
\DeclareMathOperator{\Inv}{Inv}
\DeclareMathOperator{\inv}{inv}
\DeclareMathOperator{\coinv}{coinv}
\DeclareMathOperator{\mathleg}{lg}
\DeclareMathOperator{\sort}{sort}
\DeclareMathOperator{\Par}{\mathrm{Par}}
\DeclareMathOperator{\Comp}{\mathrm{Comp}}
\DeclareMathOperator{\Compred}{\mathrm{Comp}^{\textit{red}}}

\begin{document}

\title{Stable-Limit Non-symmetric Macdonald Functions}
\author{Milo James Bechtloff Weising}
\date{\today}

\maketitle

\abstract{We construct and study an explicit simultaneous $\mathscr{Y}$-eigenbasis of Ion and Wu's standard representation of the $^+$stable-limit double affine Hecke algebra for the limit Cherednik operators $\mathscr{Y}_i$. This basis arises as a generalization of Cherednik's non-symmetric Macdonald polynomials of type $GL$. We utilize links between $^+$stable-limit double affine Hecke algebra theory of Ion-Wu and the double Dyck path algebra of Carlsson-Mellit that arose in their proof of the Shuffle Conjecture. As a consequence, the spectral theory for the limit Cherednik operators is understood. The symmetric functions comprise the zero weight space. We introduce one extra operator that commutes with the $\y_i$ action and dramatically refines the weight spaces to now be one-dimensional. This operator, up to a change of variables, gives an extension of Haiman's operator $\Delta'$ from $\Lambda$ to $\sP_{as}^{+}.$
Additionally, we develop another method to build this weight basis using limits of trivial idempotents.}

\tableofcontents

\section{Introduction}\label{intro section}
The Shuffle Conjecture \cite{HHLRU-2003}, now the Shuffle Theorem \cite{CM_2015}, is a combinatorial statement regarding the Frobenius character, $\mathcal{F}_{R_n}$, of the diagonal coinvariant algebra $R_n$ which generalizes the coinvariant algebra arising from the geometry of flag varieties. 
The conjecture built on the work of many people during the 1990s, including but not limited to Bergeron, Garsia, Haiman,  and Tesler \cite{BGSciFi} \cite{GHCat} \cite{BGHT}.
The following explicit formula is due to Haiman \cite{HVanish}
%Haiman or HHLRU??:
    $$ \mathcal{F}_{R_n}(X;q,t) = (-1)^n \nabla e_n[X] $$
    where the operator $\nabla$ is a diagonalizable operator on symmetric functions prescribed by its action on the modified Macdonald symmetric functions  $\widetilde{H}_{\mu}$  as
    $$\nabla \widetilde{H}_{\mu} = \widetilde{H}_{\mu}[-1]\cdot  \widetilde{H}_{\mu} .$$
    The original conjecture of Haglund, Haiman, Loehr, Remmel, and Ulyanov \cite{HHLRU-2003} states the following:
        
\begin{thm}[Shuffle Theorem]\cite{CM_2015}

$$ (-1)^n \nabla e_n[X] = \sum_{\pi}\sum_{ w \in \WP_{\pi}} t^{\area(\pi)} q^{\dinv(\pi,w)} x_{w}. $$

\end{thm}

In the above, $\pi$ ranges over the set of Dyck paths of length $n$ and $\WP_{\pi}$ is the set of word parking functions corresponding to $\pi$. The values $\area(\pi)$ and $\dinv(\pi,w)$ are certain statistics corresponding to $\pi$ and $w \in \WP_{\pi}$.

In \cite{CM_2015}, Carlsson and Mellit prove the Compositional Shuffle Conjecture of Haglund, Morse, and Zabrocki \cite{HMZ}, a generalization of the original Shuffle Conjecture. Carlsson and Mellit construct and investigate a quiver path algebra called the Double Dyck Path algebra $\mathbb{A}_{q,t}$. They construct a representation of $\mathbb{A}_{q,t}$, called the standard representation, built on certain mixed symmetric and non-symmetric polynomial algebras with actions from Demazure-Lusztig operators, Hall-Littlewood creation operators, and plethysms. The Compositional Shuffle Theorem 
falls out after a rich understanding of the standard representation is developed. Later analysis done by Carlsson, Gorsky, and Mellit \cite{GCM_2017} showed that in fact $\mathbb{A}_{q,t}$ occurs naturally in the context of equivariant cohomology of Hilbert schemes. 

Recent work by Ion and Wu \cite{Ion_2022} has solidified the links between the work of Carlsson and Mellit on $\mathbb{A}_{q,t}$ and the representation theory of double affine Hecke algebras. Ion and Wu introduce the $^+$stable-limit double affine Hecke algebra $\mathscr{H}^{+}$ along with a representation of $\mathscr{H}^{+}$ on the space of almost-symmetric functions, $\mathscr{P}_{as}^{+}$, from which one can recover the standard $\mathbb{A}_{q,t}$ representation. The main obstruction in making a stable-limit theory for the double affine Hecke algebras is the lack of an inverse/directed-limit system of the double affine Hecke algebras in the traditional sense. Ion and Wu get around this obstruction by introducing a new notion of convergence (Defn. \ref{defn5}) for sequences of polynomials with increasing numbers of variables along with limit versions of the standard Cherednik operators defined by this convergence. 

Central to the study of the standard Cherednik operators are the non-symmetric Macdonald polynomials. The non-symmetric Macdonald polynomials in full generality were introduced first by Cherednik \cite{C_2001} in the context of proving the Macdonald constant-term conjecture. The introduction of the double affine Hecke algebra, along with the non-symmetric Macdonald polynomials by Cherednik, constituted a significant development in representation theory. They serve as a non-symmetric counterpart to the symmetric Macdonald polynomials introduced by Macdonald as a $q,t$-analog of Schur functions. Further, they give an orthogonal basis of the polynomial representation consisting of weight vectors for the Cherednik operators. The spectral theory of non-symmetric Macdonald polynomials is well understood using the combinatorics of affine Weyl groups. The correct choice of symmetrization applied to a non-symmetric Macdonald polynomial will yield their symmetric counterpart. The type A symmetric Macdonald polynomials are a remarkable basis for symmetric polynomials simultaneously generalizing many other well studied bases which can be recovered by appropriate specializations of values for $q$ and $t$. The aforementioned modified Macdonald functions $\widetilde{H}_{\mu}$ can be obtained via a plethystic transformation from the symmetric Macdonald polynomials in sufficiently many variables.

It is natural to seek a stable-limit extension for the non-symmetric Macdonald polynomials following the methods of Ion and Wu. In particular, does the standard $\mathscr{H}^{+}$ representation $\mathscr{P}_{as}^{+}$ have a basis of weight vectors for the limit Cherednik operators $\y_i$? The first main theorem of this paper (Theorem \ref{main theorem}) answers this question in the affirmative. In the second main theorem of this paper (Theorem \ref{second main theorem}) we use a new operator $\Psi_{p_1}$, which commutes with the limit Cherednik operators, to distinguish between $\y$-weight vectors with the same $\y$-weight. The operator $\Psi_{p_1}$ is up to a change of variables an extension of Haiman's operator $\Delta'$ \cite{HMacDG} from $\Lambda$ to $\sP_{as}^{+}$ (Remark \ref{remark about delta'}). The operator $\Psi_{p_1}$ is a limit of operators from finite variable DAHAs. We conjecture (Conjecture \ref{conjecture}) that for any symmetric function $F \in \Lambda$ there is an analogous sequence of operators from finite variable DAHAs giving an analogous operator $\Psi_{F}$ on $\sP_{as}^{+}.$ If true, this conjecture would yield an action of the elliptic Hall algebra \cite{BS} \cite{SV} on $\sP_{as}^{+}$ (Remark \ref{elliptic Hall algebra}).

This paper is the full version of the author's accepted submission to FPSAC2023 \cite{MBW}.

\subsubsection*{Structure of the paper}
Section \ref{defn and notations} introduces many of the definitions and notations needed throughout this paper. In Sections \ref{stable limits of macD poly} and \ref{weight basis section}
we construct a basis of weight vectors for the limit Cherednik operators $\mathscr{Y}_i$.
Our strategy for this is the following. First, in Section \ref{stable limits of macD poly} we show that the non-symmetric Macdonald polynomials have stable-limits in the sense that if we start with a composition $\mu$ and consider the compositions $\mu * 0^m$ for $m \geq 0$ then the corresponding sequence of non-symmetric Macdonald polynomials $E_{\mu * 0^m}$ converges to an element $\widetilde{E}_{\mu}$ of $\mathscr{P}_{as}^{+}$. Next, in Section \ref{weight basis section} we show that these limits of non-symmetric Macdonald polynomials are $\mathscr{Y}$-weight vectors. Importantly, the newly constructed set of $\widetilde{E}_{\mu}$ do \textbf{\textit{not}} span $\mathscr{P}_{as}^{+}$. To fill in these gaps, the lowering operators $d_{-}$ from $\mathbb{A}_{q,t}$ are used to create enough $\y$-weight vectors to span $\mathscr{P}_{as}^{+}$. Finally, a symmetrization operator is used to show that the spanning set obtained from this process is actually a basis in Theorem \ref{main theorem}.

Lemma \ref{convergence of eigenvalues}, Corollary \ref{stable macdonald are weight vectors}, and Lemma \ref{weight vectors and lowering} together give a description of the weights for the above weight basis of  $\mathscr{P}_{as}^{+}$; in other words we describe the $\y$-spectrum. 

In the last two sections of this paper we investigate some applications of Theorem \ref{main theorem}. In 
Section \ref{recurrence relations section} we derive some recurrence relations for the stable-limit Macdonald function basis similar to the classical Knop-Sahi relations. In Section \ref{one dim wt spaces section} we construct an operator $\Psi_{p_1}$ on $\sP_{as}^{+}$ which is diagonal on the stable-limit Macdonald function basis, and thus commutes with the limit Cherednik operators $\y$.  The action of $\Psi_{p_1}$  distinguishes between our basis elements with identical $\y$-weight. This leads to the second main theorem of this paper, Theorem \ref{second main theorem}, where we prove that after adding this new operator to the algebra of limit Cherednik operators the resulting algebra is commutative and has one dimensional weight spaces in $\sP_{as}^{+}.$

\subsubsection*{Acknowledgments} 
The author would first like to thank their advisor Monica Vazirani for her tremendous help with the proof checking and editing of this paper and for her continued guidance of the author through the hurdles of the academic world. The author would like to thank the FPSAC2023 referees who alerted the author to an unpublished work of Ion and Wu \cite{IW_2022} which independently determines the $\y$-spectrum on $\mathscr{P}_{as}^{+}$.
%same explicit description of these eigenvalues. 
The author would also like to thank Daniel Orr and Ben Goodberry for their insights regarding the symmetrization operators occurring in subsection \ref{symmetrization subsection}.

\section{Definitions and Notation}\label{defn and notations}

\subsection{Double Affine Hecke Algebras in Type GL}

We present here the conventions that will be used in this paper for the double affine Hecke algebra of type $GL$. Take note of the quadratic relation $(T_{i}-1)(T_{i}+t) = 0$ which has been chosen to match with the conventions in \cite{Ion_2022} but may differ from other authors.

\clearpage
\begin{defn} \label{defn1}
Define the \textbf{\textit{double affine Hecke algebra}} $\mathscr{H}_n$ to be the $\mathbb{Q}(q,t)$-algebra generated by $T_1,\ldots,T_{n-1}$, $X_1^{\pm 1},\ldots,X_{n}^{\pm 1}$, and $Y_1^{\pm 1},\ldots,Y_n^{\pm 1}$ with the following relations:

\begin{multicols}{2}
\begin{enumerate}[(i)]
%\begin{itemize}
    \item %[(i)]  
    \label{def-i} 
    $(T_i -1)(T_i +t) = 0$,
    \item [] $T_iT_{i+1}T_i = T_{i+1}T_iT_{i+1}$,
    \item [] $T_iT_j = T_jT_i$, $|i-j|>1$,
    \item %[(ii)] 
    \label{def-ii} $T_i^{-1}X_iT_i^{-1} = t^{-1}X_{i+1}$,
    \item []$T_iX_j = X_jT_i$, $j \notin \{i,i+1\}$,
    \item []$X_iX_j = X_jX_i$,
    \item %[(iii)]
    \label{def-iii}$T_iY_iT_i = tY_{i+1}$,
    \item []$T_iY_j = Y_jT_i$, $j\notin \{i,i+1\}$,
    \item []$Y_iY_j = Y_jY_i$,
    \item %[(iv)]
    \label{def-iv}$Y_1T_1X_1 = X_2Y_1T_1$,
    \item %[(v)]
    \label{def-v}$Y_1X_1\cdots X_n = qX_1\cdots X_nY_1$
    \item []
%\end{itemize}
\end{enumerate}
\end{multicols}

Further, define the special element $\omega_n$ by $$\omega_n := T_{n-1}^{-1}\cdots T_1^{-1}Y_1^{-1}.$$ This conveniently allows us to write $$Y_1 = \omega_n^{-1}T_{n-1}^{-1}\cdots T_1^{-1}.$$ When required we will write $Y_i^{(n)}$ for the element $Y_i$ in $\mathscr{H}_n$ to differentiate between the element $Y_i^{(m)}$ in a different $\mathscr{H}_m$ for $n \neq m$.

We will often use the following basic fact about $\mathscr{H}_n$ the proof of which we will omit. 

\begin{lem} \label{sym x and y commute with T's}
    Let $f(X_1,\ldots,X_n), g(Y_1,\ldots, Y_n) \in \mathscr{H}_n$ be symmetric Laurent polynomials in $X$'s and $Y$'s respectively. Then for all $1 \leq i \leq n-1$, 
    $$[T_i,f(X_1,\ldots,X_n)] = [T_i,g(Y_1,\ldots, Y_n)] = 0.$$
\end{lem}

\end{defn}

\subsubsection{Standard DAHA Representation}
\begin{defn} \label{defn2}
Let $\mathscr{P}_n = \mathbb{Q}(q,t)[x_1^{\pm 1},\ldots,x_n^{\pm 1}]$. The \textbf{\textit{ standard representation}} of $\mathscr{H}_n$ is given by the following action on $\mathscr{P}_n$:

\begin{center}
\begin{itemize}
    \item $T_if(x_1,\ldots,x_n) = s_i f(x_1,\ldots,x_n) +(1-t)x_i \frac{1-s_i}{x_i-x_{i+1}}f(x_1,\ldots,x_n)$
    \item $X_if(x_1,..,x_n)= x_if(x_1,\ldots,x_n)$
    \item $\omega_nf(x_1,\ldots,x_n) = f(q^{-1}x_n,x_1,\ldots,x_{n-1})$
\end{itemize}
\end{center}

Here $s_i$ denotes the operator that swaps the variables $x_i$ and $x_{i+1}$. Under this action the $T_i$ operators are known as the \textbf{\textit{Demazure-Lusztig operators}}. For $q,t$ generic $\mathscr{P}_n$ is known to be a faithful representation of $\mathscr{H}_n$. The action of the elements $Y_1,\ldots,Y_n \in \mathscr{H}_n$ are called \textbf{\textit{Cherednik operators}}.
\end{defn}

Set $\mathscr{H}_n^{+}$ to be the positive part of $\mathscr{H}_n$ i.e. the subalgebra generated by $T_1,\ldots,T_{n-1}$, $X_1,\ldots,X_n$, and $Y_1,\ldots,Y_n$ without allowing for inverses in the $X$ and $Y$ elements and set $\mathscr{P}_n^{+} = \mathbb{Q}(q,t)[x_1,\ldots,x_n]$. Importantly, 
$\mathscr{P}_n^{+}$ is a $\mathscr{H}_n^{+}$-submodule of $\mathscr{P}_n$. 

\begin{defn}
    For $1\leq i \leq n-1$ define the intertwiners, $\varphi_i^{(n)} \in \mathscr{H}_n$, as 
    $$\varphi_i^{(n)}:= [T_i,Y_i^{(n)}] = T_iY_i^{(n)} - Y_i^{(n)}T_i.$$
\end{defn}

The intertwiner elements have the following properties which are readily verified from the relations of $\mathscr{H}_{n}$: 
\begin{itemize}
    \item $\varphi_i^{(n)} = T_i(Y_i^{(n)}-Y_{i+1}^{(n)}) + (1-t)Y_{i+1}^{(n)}$
    \item $\varphi_i^{(n)}Y_j^{(n)} = Y_{s_{i}(j)}^{(n)}\varphi_i^{(n)}$
    \item $(\varphi_i^{(n)})^2 = (Y_i^{(n)}-tY_{i+1}^{(n)})(Y_{i+1}^{(n)}-tY_i^{(n)}).$
\end{itemize}

\subsubsection{Non-symmetric Macdonald Polynomials}
Before discussing non-symmetric Macdonald polynomials we must first review some basic combinatorial definitions.

\begin{defn}
 In this paper, a \textbf{\textit{composition}} will refer to a finite tuple $\mu = (\mu_1,\ldots,\mu_n)$ of non-negative integers. We allow for the empty composition $\emptyset$ with no parts. We will let $\Comp$ denote the set of all compositions. The length of a composition $\mu = (\mu_1,\ldots,\mu_n)$ is $\ell(\mu) = n$ and the size of the composition is $| \mu | = \mu_1+\ldots+\mu_n$. As a convention we will set $\ell(\emptyset) = 0$ and $|\emptyset| = 0.$ We say that a composition $\mu$ is \textbf{\textit{reduced}} if $\mu = \emptyset$ or $\mu_{\ell(\mu)} \neq 0.$ We will let $\Compred$ denote the set of all reduced compositions. Given two compositions $\mu = (\mu_1,\ldots,\mu_n)$ and $\beta = (\beta_1,\ldots,\beta_m)$, define $\mu * \beta = (\mu_1,\ldots,\mu_n,\beta_1,\ldots,\beta_m)$. A \textbf{\textit{partition}} is a composition $\lambda = (\lambda_1,\ldots,\lambda_n)$ with $\lambda_1\geq \ldots \geq \lambda_n \geq 1$. Note that vacuously we allow for the empty partition $\emptyset.$ We denote the set of all partitions by $\Par$. We denote $\sort(\mu)$ to be the partition obtained by ordering the nonzero elements of $\mu$ in weakly decreasing order. The dominance ordering for partitions is defined by $\lambda \trianglelefteq \nu$ if for all $i\geq 1$, $\lambda_1 + \ldots +\lambda_i \leq \nu_1 + \ldots +\nu_i$ where we set $\lambda_i = 0$ whenever $i > \ell(\lambda)$ and similarly for $\nu$. If $\lambda \trianglelefteq \nu$ and $\lambda \neq \nu.$ we will write $\lambda \triangleleft \nu$.
 
 We will in a few instances use the notation $\mathbbm{1}(p)$ to denote the value $1$ if the statement p is true and $0$ otherwise. In this paper we will write $\mathfrak{S}_n$ for the permutation group on the set $[n]:= \{1,\ldots,n\}.$
 \end{defn}

In line with the conventions in \cite{haglund2007combinatorial} we define the Bruhat order on the type $GL_n$ weight lattice $\mathbb{Z}^{n}$ as follows. 

\begin{defn}
Let $e_1,...,e_n$ be the standard basis of $\mathbb{Z}^n$ and let $\alpha \in \mathbb{Z}^n$. We define the\textbf{\textit{ Bruhat ordering}} on $\mathbb{Z}^n$, written simply by $<$, by first defining cover relations for the ordering and then taking their transitive closure. If $i<j$ such that $\alpha_i < \alpha_j$ then we say $\alpha > (ij)(\alpha)$ and additionally if $\alpha_j - \alpha_i > 1$ then $(ij)(\alpha) > \alpha + e_i - e_j$ where $(ij)$ denotes the transposition swapping $i$ and $j.$
\end{defn}

As an equivalent definition we say $\alpha < \beta$ if $\sort(\alpha) \triangleleft \sort(\beta)$ and in the case that $\lambda = \sort(\alpha) = \sort(\beta)$ then we have $\alpha < \beta$ when $\sigma < \gamma$ in the Bruhat order for minimal length permutations $\sigma, \gamma$ with $\sigma(\alpha) = \lambda$, $\gamma(\beta) = \lambda$. It is important to note that with respect to the Bruhat order any weakly decreasing vector $v \in \mathbb{Z}^n$ is the minimal element in its permutation orbit $\mathfrak{S}_n.v$.

\begin{defn} \label{defn3}
The \textbf{\textit{non-symmetric Macdonald polynomials}} (for $GL_n$) are a family of Laurent polynomials $E_{\mu} \in \mathscr{P}_n$ for $\mu \in \mathbb{Z}^n$ uniquely determined by the following:

\begin{itemize}
    \item Triangularity: Each $E_{\mu}$ has a monomial expansion of the form $E_{\mu} = x^{\mu} + \sum_{\lambda < \mu} a_{\lambda}x^{\lambda}$

    \item Weight Vector: Each  $E_{\mu}$ is a weight vector for the operators $Y_1^{(n)},\ldots,Y_n^{(n)} \in \mathscr{H}_n$.
\end{itemize}
\end{defn}

The non-symmetric Macdonald polynomials are a $Y^{(n)}$-weight basis for the $\mathscr{H}_n$ standard representation $\mathscr{P}_n$. For $\mu \in \mathbb{Z}^n$, $E_{\mu}$ is homogeneous with degree $\mu_1+\ldots +\mu_n$. Further, the set of $E_{\mu}$ corresponding to $\mu \in \mathbb{Z}_{\geq 0}^n$ gives a basis for $\mathscr{P}_n ^{+}$.

\subsubsection{Combinatorial Formula for Non-symmetric Macdonald Polynomials}

Note that the $q,t$ conventions in \cite{haglund2007combinatorial} differ from those appearing in this paper. In the below theorem the appropriate translation $q \rightarrow q^{-1}$ has been made.

In \cite{haglund2007combinatorial}, Haglund, Haiman, and Loehr give an explicit monomial expansion formula for the non-symmetric Macdonald polynomials in terms of the combinatorics of \textbf{\textit{non-attacking labellings}} of certain box diagrams corresponding to compositions which we will now review.

\begin{defn}\cite{haglund2007combinatorial}
For a composition $\mu = (\mu_1,\dots,\mu_n)$ define the column diagram of $\mu$ as 
$$dg'(\mu):= \{(i,j)\in \mathbb{N}^2 : 1\leq i\leq n, 1\leq j \leq \mu_i \}.$$ This is represented by a collection of boxes in positions given by $dg'(\mu)$. The augmented diagram of $\mu$ is given by 
$$\widehat{dg}(\mu):= dg'(\mu)\cup\{(i,0): 1\leq i\leq n\}.$$
Visually, to get $\widehat{dg}(\mu)$ we are adding a bottom row of boxes on length $n$ below the diagram $dg'(\mu)$. Given $u = (i,j) \in dg'(\mu)$ define the following:
\begin{itemize}
    \item $\leg(u) := \{(i,j') \in dg'(\mu): j' > j\}$
    \item $\arm^{\mathleft}(u) := \{(i',j) \in dg'(\mu): i'<i, \mu_{i'} \leq \mu_i\} $
    \item $\arm^{\mathright}(u):= \{(i',j-1) \in \widehat{dg}(\mu): i'>i, \mu_{i'}<\mu_{i}\}$
    \item $\arm(u) := \arm^{\mathleft}(u) \cup \arm^{\mathright}(u)$
    \item $\mathleg(u):= |\leg(u)| = \mu_i -j$
    \item $a(u) := |\arm(u)|.$
\end{itemize}
A filling of $\mu$ is a function $\sigma: dg'(\mu) \rightarrow \{1,...,n\}$ and given a filling there is an associated augmented filling $\widehat{\sigma}: \widehat{dg}(\mu) \rightarrow \{1,...,n\}$ extending $\sigma$ with the additional bottom row boxes filled according to $\widehat{\sigma}((j,0)) = j$ for $j = 1,\dots,n$. Distinct lattice squares $u,v \in \mathbb{N}^2$ are said to attack each other if one of the following is true:
\begin{itemize}
\item $u$ and $v$ are in the same row 
\item $u$ and $v$ are in consecutive rows and the box in the lower row is to the right of the box in the upper row.
\end{itemize}
A filling $\sigma: dg'(\mu) \rightarrow \{1,\dots,n\}$ is non-attacking if $\widehat{\sigma}(u) \neq \widehat{\sigma}(v)$ for every pair of attacking boxes $u,v \in \widehat{dg}(\mu).$
For a box $u= (i,j)$ let $d(u) = (i,j-1)$ denote the box just below $u$. Given a filling $\sigma:dg'(\mu)\rightarrow \{1,\dots,n\}$, a descent of $\sigma$ is a box $u \in dg'(\mu)$ such that $\widehat{\sigma}(u) > \widehat{\sigma}(d(u)).$
Set $\Des(\widehat{\sigma})$ to be the set of descents of $\widehat{\sigma}$ and define 
$$\maj(\widehat{\sigma}):= \sum_{u \in \Des(\widehat{\sigma})} (\mathleg(u)+1).$$
The reading order on the diagram $\widehat{dg}(\mu)$ is the total ordering on the boxes of $\widehat{dg}(\mu)$ row by row, from top to bottom, and from right to left within each row. If $\sigma: dg'(\mu) \rightarrow \{1,\dots,n\}$ is a filling, an inversion of $\widehat{\sigma}$ is a pair of attacking boxes $u,v \in \widehat{dg}(\mu)$ such that $u < v$ in reading order and $\widehat{\sigma}(u) > \widehat{\sigma}(v).$ Set $\Inv(\widehat{\sigma})$ to be the set of inversions of $\widehat{\sigma}$. Define the statistics 
\begin{itemize}
    \item $\inv(\widehat{\sigma}):= |\Inv(\widehat{\sigma})| -|\{i<j: \mu_i \leq \mu_j\}| - \sum_{u \in \Des(\widehat{\sigma})} a(u)$
    \item $\coinv(\widehat{\sigma}):= \left( \sum_{u \in dg'(\mu)} a(u) \right) -\inv(\widehat{\sigma}).$
    \end{itemize}
Lastly, for a filling $\sigma:dg'(\mu) \rightarrow \{1,\dots,n\}$ set $$x^{\sigma}:= x_1^{|\sigma^{-1}(1)|}\cdots x_n^{|\sigma^{-1}(n)|}.$$
\end{defn}

The combinatorial formula for non-symmetric Macdonald polynomials can now be stated.

\begin{thm}\cite{haglund2007combinatorial} \label{HHL}
For a composition $\mu$ with $\ell(\mu) = n$ the following holds:
    $$E_{\mu} = \sum_{\substack{\sigma: \mu \rightarrow [n]\\ \text{non-attacking}}} x^{\sigma}q^{-\maj(\widehat{\sigma})}t^{\coinv(\widehat{\sigma})} \prod_{\substack{u \in dg'(\mu) \\ \widehat{\sigma}(u) \neq \widehat{\sigma}(d(u))}} \left( \frac{1-t}{1-q^{-(\mathleg(u)+1)}t^{(a(u)+1)}} \right). $$
\end{thm}

\begin{example}
We finish this subsection with a visual example of a non-attacking filling and its associated statistics. Below is the augmented filling $\widehat{\sigma}$ of a non-attacking filling $\sigma: (3,2,0,1,0,0) \rightarrow [6]$ pictured as labels inside the boxes of $\widehat{dg}(3,2,0,1,0,0).$

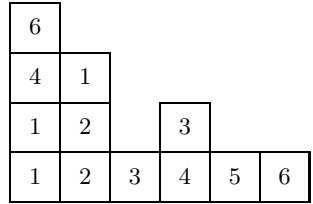
\begin{center}

\ytableausetup{centertableaux, boxframe= normal, boxsize= 2em}
\begin{ytableau}
6 & \none & \none & \none & \none & \none \\
4 &   1   & \none & \none & \none & \none \\
1 &   2   & \none &   3   & \none & \none \\
1 &   2   &   3   &   4   &   5   &   6   \\
\end{ytableau}

\end{center}

Let $u$ be the column 1 box of $\widehat{dg}(3,2,0,1,0,0)$ filled with a $4$ in the above diagram. Notice that $u$ is a descent box of $\widehat{\sigma}$ as $4$ is larger than the label $1$ of the box $d(u)$ just below $u.$ Further, we see that $a(u) = 2$ and $\mathleg(u) = 1$. Considering the diagram as a whole now we see that $x^{\sigma} = x_1^{2}x_2x_3x_4x_6$, $\maj(\widehat{\sigma}) = 3$, $|\Inv(\widehat{\sigma})| = 21$, $\inv(\widehat{\sigma}) = 14$, and $\coinv(\widehat{\sigma}) = 1.$ The contribution of this non-attacking labelling to the HHL formula for $E_{(3,2,0,1,0,0)} \in \sP_{6}^{+}$ is 
$$x_1^2x_2x_3x_4x_6 q^{-3}t^{1} \left( \frac{1-t}{1-q^{-1}t^3} \right)\left( \frac{1-t}{1-q^{-1}t^2} \right)\left( \frac{1-t}{1-q^{-2}t^{3}} \right)\left( \frac{1-t}{1-q^{-1}t^2} \right).$$
\end{example}

\subsection{Symmetric Functions}

 \begin{defn}
 Define the \textbf{\textit{ring of symmetric functions}} $\Lambda$ to be the subalgebra of the inverse limit of the symmetric polynomial rings $\mathbb{Q}(q,t)[x_1,\ldots,x_n]^{\mathfrak{S}_n}$ with respect to the quotient maps sending $x_n \rightarrow 0$ consisting of those elements with bounded $x$-degree. For $i \geq 0$ define the $i$-th \textbf{\textit{power sum symmetric function}} by
 $$p_i = x_1^i + x_2^i + \dots .$$ It is a classical result that $\Lambda$ is isomorphic to $\mathbb{Q}(q,t)[p_1,p_2,\dots]$. For any expression $G = a_1g^{\mu_1} + a_2g^{\mu_2} +\dots$ with rational scalars $a_i \in \mathbb{Q}$ and distinct monomials $g^{\mu_i}$ in a set of algebraically independent commuting free variables $\{g_1,g_2,\dots\}$ the \textbf{\textit{plethsytic evaluation}} of $p_i$ at the expression $G$ is defined to be $$p_i[G] := a_1g^{i\mu_1} +a_2g^{i \mu_2}+ \dots .$$ Note that $g_i$ are allowed to be $q$ or $t.$ Here we are using the convention that $i\mu = (i\mu_1,\ldots, i \mu_r)$ for $\mu = (\mu_1,\cdots, \mu_r).$ The definition of plethystic evaluation on power sum symmetric functions extends to all symmetric functions $F \in \Lambda$ by requiring $F \rightarrow F[G]$ be a $\mathbb{Q}(q,t)$-algebra homomorphism. Note that for $F \in \Lambda$, $F = F[x_1+x_2+\ldots]$ and so we will often write $F = F[X]$ where $X:= x_1+ x_2 + \ldots.$ For a partition $\lambda$ define the \textbf{\textit{monomial symmetric function}} $m_\lambda$ by 
 $$m_\lambda := \sum_{\mu} x^{\mu}$$
 where we range over all distinct monomials $x^{\mu}$ such that $\sigma(\mu) = \lambda$ for some permutation $\sigma$. For $n \geq 0$ define the \textbf{\textit{complete homogeneous symmetric function}} $h_n$ by 
 $$h_n:= \sum_{|\lambda|= n} m_\lambda .$$ We can extend plethysm to $\mathbb{Q}(q,t)[[p_1,p_2,\dots]]$.  The \textbf{\textit{plethystic exponential}} is defined to be the element of $\mathbb{Q}(q,t)[[p_1,p_2,\dots]]$ given by 
 $$Exp[X]:= \sum_{n \geq 0} h_n[X].$$ 
 \end{defn}

Here we list some notable properties of the plethystic exponential which will be used later in this paper.

\begin{itemize}
    \item $Exp[0] = 1$
    \item $Exp[X+Y]= Exp[X]Exp[Y]$
    \item $Exp[x_1+x_2+\dots] = \prod_{i=1}^{\infty} \left( \frac{1}{1-x_i} \right)$
    \item $Exp[(1-t)(x_1+x_2+\dots)] = \prod_{i=1}^{\infty} \left( \frac{1-tx_i}{1-x_i} \right)$
\end{itemize}

\begin{example}
Here we give a few examples of plethystic evaluation.
    \begin{itemize}
        \item $p_3[1+5t+qt^2] = 1+5t^3 + q^3t^6$
        \item $s_{2}[(1-t)X] = (\frac{p_2+p_{1,1}}{2})[(1-t)X] = \frac{(1-t^2)p_2[X]+(1-t)^2p_{1,1}[X]}{2}$
        \item $Exp[\frac{t}{1-t}]= \prod_{n=1}^{\infty}(\frac{1}{1-t^n})$
    \end{itemize}
\end{example}

\subsubsection{Hall-Littlewood Symmetric Functions}
For the purposes of this paper we need the following explicit collection of symmetric functions.
\begin{defn}
    For $n \geq 0$ define the \textbf{\textit{Jing vertex operator}} $\mathscr{B}_n \in End_{\mathbb{Q}(q,t)}(\Lambda)$ by
    $$\mathscr{B}_n[F] : = \langle z^n \rangle F[X-z^{-1}]Exp[(1-t)zX].$$
    Here $\langle z^n \rangle $ denotes the operator which extracts the coefficient of $z^n$ of any formal series in $z$. For a partition $\lambda = (\lambda_1,...,\lambda_r)$ define the \textbf{\textit{Hall-Littlewood symmetric function}}, $\mathcal{P}_{\lambda}$, by 
    $$\mathcal{P}_{\lambda}:= \mathscr{B}_{\lambda_1}\cdots \mathscr{B}_{\lambda_r}(1).$$
\end{defn}

Note that the operator $\mathscr{B}_n$ is graded with degree $n.$ The definition of the Hall-Littlewood symmetric functions in this paper matches with \cite{Ion_2022} and \cite{CM_2015} but differs from that of other authors. As we will see later in Proposition \ref{idempotent = HL} the $\mathcal{P}_{\lambda}[X]$ are the same as the dual Hall-Littlewood symmetric functions $Q_{\lambda}[X;t]$ defined by Macdonald \cite{Macdonald}. These symmetric functions have the following useful properties.
\begin{itemize}
    \item $\mathcal{P}_{\lambda}$ is homogeneous with degree $|\lambda|$
    \item $\mathcal{P}_{(n)}[X] = h_{n}[(1-t)X]$
    \item If $n \geq \lambda_1$ then $\mathscr{B}_{n}(\mathcal{P}_{\lambda}) = \mathcal{P}_{n*\lambda}$
    \item $\mathscr{B}_{0}(\mathcal{P}_{\lambda}) = t^{\ell(\lambda)} \mathcal{P}_{\lambda}$
\end{itemize}
Lastly, it is a classical result that the collection $\{\mathcal{P}_\lambda \mid \lambda \in \Par\}$ is a basis of $\Lambda$.

\subsection{Stable-Limit DAHA of Ion and Wu}
As the index $n$ varies, the standard $\mathscr{H}_n$ representations, $\sP_n$, fail to form a direct/inverse system of compatible $\mathscr{H}_n$ representations. However, as the authors Ion and Wu investigate in \cite{Ion_2022}, this sequence of representations is compatible enough to allow for the construction of a limiting representation for a new algebra resembling a direct limit of the double affine Hecke algebras of type $GL$. We will start by giving the definition of this algebra.

\begin{defn}\cite{Ion_2022} \label{defn4}
The \textit{\textbf{$^+$stable-limit double affine Hecke algebra}} of Ion and Wu, $\mathscr{H}^{+}$, is the algebra  generated over $\mathbb{Q}(q,t)$ by the elements $T_i,X_i,Y_i$ for $i \geq 1$ satisfying the following relations:

\begin{center}
\begin{itemize}
    \item The generators $T_i,X_i$ for $i \in \mathbb{N}$ satisfy %(i)
    \eqref{def-i}  and \eqref{def-ii}
    %(ii)
    of Defn. \ref{defn1}.
    \item The generators $T_i,Y_i$ for $i \in \mathbb{N}$ satisfy 
    %(i) and (iii)
    \eqref{def-i} and \eqref{def-iii}
    of Defn. \ref{defn1}.
    \item  $Y_1T_1X_1 = X_2Y_1T_1.$
 \end{itemize}
 \end{center}
 \end{defn}

Importantly, there is no relation of the form $Y_1X_1\cdots X_n = q X_1\cdots X_n Y_1$ in $\mathscr{H}^{+}$. As such there is no invertible '$\omega$' element in $\mathscr{H}^{+}$ which in $\mathscr{H}_n$ normally realizes the cyclic symmetry of the affine type $A$ root systems.

\begin{defn}\cite{Ion_2022} \label{defn5}
 Let $\mathscr{P}_{\infty}^{+}$ denote the inverse limit of the rings $\mathscr{P}_{k}^{+}$ with respect to the homomorphisms $\pi_k: \sP_{k+1}^{+} \rightarrow \sP_k^{+}$ which send $x_{k+1}$ to 0 at each step. We can naturally extend $\pi_k$ to a map $\mathscr{P}_{\infty}^{+} \rightarrow \mathscr{P}_k$ which will be given the same name. Let $\mathscr{P}(k)^{+} := \mathbb{Q}(q,t)[x_1,\ldots,x_k]\otimes \Lambda[x_{k+1}+x_{k+2}+\ldots]$. Define the \textbf{\textit{ring of almost symmetric functions}} by $\mathscr{P}_{as}^{+} := \bigcup_{k\geq 0} \mathscr{P}(k)^{+}$. Note $\mathscr{P}_{as}^{+} \subset \mathscr{P}_{\infty}^{+}.$ Define $\rho: \mathscr{P}_{as}^{+} \rightarrow x_1\mathscr{P}_{as}^{+}$ to be the linear map defined by $\rho(x_1^{a_1}\cdots x_n^{a_n}F[x_{m}+x_{m+1}+\ldots]) = \mathbbm{1}(a_1 > 0) x_1^{a_1}\cdots x_n^{a_n}F[x_{m}+x_{m+1}+\ldots] $ for $F \in \Lambda$. Note that $\rho$ restricts to maps $\sP_n \rightarrow x_1\sP_n$ which are compatible with the quotient maps $\pi_n$.
 \end{defn}

The ring $\sP_{as}^{+}$ is a free graded $\Lambda$-module with homogeneous basis given simply by the set of monomials $x^{\mu}$ with $\mu$ reduced. Therefore, $\sP_{as}^{+}$ has the homogeneous $\mathbb{Q}(q,t)$ basis given by all $x^{\mu}m_{\lambda}[X]$ ranging over all reduced compositions $\mu$ and partitions $\lambda$. Further, the dimension of the homogeneous degree d part of $\mathscr{P}(k)^{+}$ is equal to the number of pairs $(\mu,\lambda)$ of reduced compositions $\mu$ and partitions $\lambda$ with $|\mu|+|\lambda| = d$ and $\ell(\mu) \leq k$. 
 
In order to define the operators required for Ion and Wu's main construction we must first review the new definition of convergence introduced in \cite{Ion_2022}. 

 \begin{defn} \cite{Ion_2022}\label{defn6}
 Let $(f_m)_{m \geq 1}$ be a sequence of polynomials with $f_m \in \mathscr{P}_m^{+}$. Then the sequence $(f_m)_{m \geq 1}$ is \textbf{\textit{convergent}} if there exist some N and auxiliary sequences $(h_m)_{m\geq1}$, $(g^{(i)}_m)_{m\geq 1}$, and $(a^{(i)}_m)_{m\geq 1}$ for $1\leq i \leq N$ with $h_m, g^{(i)}_m \in \mathscr{P}_{m}^{+}$, $a^{(i)}_m \in \mathbb{Q}(q,t)$ with the following properties:

\begin{itemize}
    \item For all $m$, $f_m = h_m + \sum_{i=1}^{N} a^{(i)}_m g^{(i)}_m$.
    \item The sequences $(h_m)_{m\geq1}$, $(g^{(i)}_m)_{m\geq1}$ for $1\leq i \leq N$ converge in $\mathscr{P}_{\infty}^{+}$ with limits $h,g^{(i)}$ respectively. That is to say, $\pi_m(h_{m+1}) = h_m$ and $\pi_m(g_{m+1}^{(i)}) = g_{m}^{(i)}$ for all $1\leq i \leq N$ and $m \geq 1$. Further, we require $g^{(i)} \in \mathscr{P}_{as}^{+}$.
    \item The sequences $a^{(i)}_m$ for $1\leq i \leq N$ converge with respect to the t-adic topology on $\mathbb{Q}(q,t)$ with limits $a^{(i)}$ which are required to be in $\mathbb{Q}(q,t)$.
\end{itemize}

The sequence is said to have a limit given by
$\lim_{m} f_m = h + \sum_{i=1}^{N}a^{(i)}g^{(i)}.$

\end{defn}

This definition of convergence is a mix of both the stronger topology arising from the inverse system given by the maps $\pi_m$ and the t-adic topology arising from the ring $\mathbb{Q}(q,t)$. It is important to note that part of the above definition requires convergent sequences to always be written as a finite sum of fixed length with terms that converge independently. 

Here we list a few instructive examples of convergent sequences and their limits:

\begin{itemize}
    \item $\lim_{m} t^m = 0$
    \item $\lim_{m} 1+\ldots +t^m = \frac{1}{1-t}$
    \item $\lim_{m} \frac{1}{q^2-t^m}(x_3^2 +\ldots+ x_{m}^2 ) = q^{-2} p_2[x_3+\ldots].$
\end{itemize}

\begin{remark}\label{useful defn of convergence}
    In this paper we will be entirely concerned with convergent sequences $(f_m)_{m \geq 1}$ with almost symmetric limits $\lim_{m} f_m \in \sP_{as}^{+}$. In this case it follows readily from definition that each of these convergent sequences necessarily will have the form 
    $$f_m(x_1,\ldots,x_m) = \sum_{i=1}^{N} c_i^{(m)}x^{\mu^{(i)}}F_i[x_1+\ldots + x_m]$$
    where $N \geq 1$ is fixed, $c_i^{(m)}$ are convergent sequences of scalars with $\lim_{m} c_i^{(m)} \in \mathbb{Q}(q,t)$, $F_i$ are symmetric functions, and $\mu^{(i)}$ are compositions. Here we will consider $x^{\mu^{(i)}} = 0$ in $\sP_{m}$ whenever $\ell(\mu^{(i)}) > m$.
\end{remark}

\begin{defn} \cite{Ion_2022}
    For $m \geq 1$ suppose $A_m$ is an operator on $\sP_m^{+}$. The sequence $(A_m)_{m \geq 1}$ of operators is said to \textbf{\textit{converge}} if for every $f \in \sP_{as}^{+}$ the sequence $(A_m(\pi_m(f)))_{m \geq 1}$ converges to an element of $\sP_{as}^{+}$. From \cite{Ion_2022} the corresponding operator on $\sP_{as}^{+}$ given by $A(f): = \lim_{m} A_m(\pi_m(f))$ is well defined and said to be the limit of the sequence $(A_m)_{m \geq 1}$. In this case we will simply write $A = \lim_{m} A_m$.
\end{defn}

There are two important examples of convergent operator sequences which will be relevant for the rest of this paper. For all $i 
\geq 1$ and $m \geq 1$ let $X_i^{(m)}$ denote the operator on $\sP_m^{+}$ given by $0$ if $m < i$ and by $X_i^{(m)}f = x_i f$ if $i \leq m$. Similarly for $i \geq 1$ and $m \geq 1$ let $T_i^{(m)}$ denote the operator on $\sP_m^{+}$ given by $0$ if $m-1 < i$ and by $T_if = s_if + (1-t)x_i \frac{f-s_if}{x_i-x_{i+1}}$ if $i \leq m-1$. Then for all $i \geq 1$ it is immediate from definition that the sequences $(X_i^{(m)})_{m \geq 1}$ and $(T_i^{(m)})_{m \geq 1}$ converge to operators $X_i$ and $T_i$ respectively on $\sP_{as}^{+}$. Further, their corresponding actions are given for $f \in \sP_{as}^{+}$ simply by 
\begin{itemize}
    \item $X_i(f) = x_if$
    \item $T_i(f) = s_if + (1-t)x_i \frac{f-s_if}{x_i-x_{i+1}}.$ 
\end{itemize}

The following important technical proposition of Ion and Wu will be used repeatedly in this paper. 

\begin{prop}[Prop. 6.21 \cite{Ion_2022}]
    If $A = \lim_{m} A_m$ and $f = \lim_{m} f_m$ are limit operators and limit functions respectively then $A(f) = \lim_{m} A_m(f_m).$
\end{prop}

This is a sort of continuity statement for convergent sequences of operators. The utility of the above proposition is that for an operator arising as the limit of finite variable operators, $A = \lim_{m} A_m$ say, we can use \textbf{\textit{any}} sequence $(f_m)_{m \geq 1}$ converging to $f \in \sP_{as}^{+}$ in order to calculate $A(f)$. 

\subsubsection{The Standard +Stable-Limit DAHA Representation}

Ion and Wu begin their construction of the standard representation of $\mathscr{H}^{+}$ by noting the following key fact. 

\begin{prop}\cite{Ion_2022}
    For $n \geq 1$  
    $$\pi_{n-1} t^n Y_1^{(n)}X_1 = t^{n-1} Y_1^{(n-1)}X_1\pi_{n-1}.$$
\end{prop}

In other words, the action of the operators $t^nY_1^{(n)}$ and $t^{n-1}Y_1^{(n-1)}$ are compatible on $x_1\mathscr{P}_{n}$. As such there exists a limit operator $Y_1^{(\infty)}: x_1\mathscr{P}^{+}_{\infty} \rightarrow x_1\mathscr{P}^{+}_{\infty}$ such that $\pi_nY_1^{(\infty)} = t^nY_1^{(n)}$. A crucial idea of Ion and Wu is to extend the action of the operators $t^n Y_1^{(n)}$ on $x_1\mathscr{P}_{n}$ to all of $\mathscr{P}_{n}$ using the previously defined projection $\rho:\mathscr{P}_{n} \rightarrow x_1\mathscr{P}_{n}$.

\begin{defn}\cite{Ion_2022} \label{defn of deformed y's}
    Define the operator $\widetilde{Y}_1^{(n)} :=  \rho \circ t^n Y_1^{(n)}$. For $2 \leq i \leq n$ define $\widetilde{Y}_i^{(n)}$ by requiring $\widetilde{Y}_i^{(n)} = t^{-1}T_{i-1}\widetilde{Y}_{i-1}^{(n)}T_{i-1}.$
\end{defn}

A direct check shows that $\widetilde{Y}_1^{(n)}X_1 = t^n Y_1^{(n)}X_1$ so that $\widetilde{Y}_1^{(n)}$ extends the action of $t^n Y_1^{(n)}$ on $x_1\mathscr{P}_{n}$ as desired. The main utility of this specific choice of definition is the following theorem.

\begin{thm}\label{deformed Y's converge}\cite{Ion_2022}
    The sequence $(\widetilde{Y}_1^{(m)})_{m \geq 1}$ converges to an operator $\mathscr{Y}_1$ on $\sP_{as}^{+}$. Define the operators $\mathscr{Y}_i$ for $i \geq 2$ by $\y_{i}:= t^{-1}T_{i-1}\y_{i-1}T_{i-1}$. The operators $\mathscr{Y}_i$ along with the Demazure-Lusztig action of the $T_i$'s and multiplication by the $X_i$'s generate an $\mathscr{H}^{+}$ action on $\mathscr{P}_{as}^{+}$.
\end{thm}

In particular, the authors Ion and Wu show that despite the fact that for $1 \leq i \neq j \leq n$, $\widetilde{Y}_i^{(n)}\widetilde{Y}_j^{(n)} \neq \widetilde{Y}_j^{(n)}\widetilde{Y}_i^{(n)}$ the limit Cherednik operators commute: $$\mathscr{Y}_i\mathscr{Y}_j = \mathscr{Y}_j\mathscr{Y}_i .$$
The action of the $\y_i$ operators respect the canonical filtration of $\sP_{as}^{+} = \bigcup_{k\geq 0} \sP(k)^{+}$. For all $n\geq 0$, the operators $\{\y_1,...,\y_n\}$ restrict to operators on the space $\sP(n)^{+}$ whereas the operators $\{\y_{n+1},\y_{n+2},...\}$ annihilate $\sP(n)^{+}$. Note that for $n = 0$, $\sP(0)^{+} = \Lambda$ so all of the operators $\y_i$ annihilate $\Lambda$.

\subsection{Double Dyck Path Algebra}
The \textbf{\textit{Double Dyck Path Algebra}} $\mathbb{A}_{q,t}$, introduced by Carlsson and Mellit \cite{CM_2015}, is a quiver path algebra with vertices indexed by non-negative integers with the following edge operators:
 
 \begin{center}
\begin{varwidth}{\textwidth}
\begin{itemize}
\setlength\itemsep{-0.5em}
    \item $d_{+},d_{+}^{*}: k \rightarrow k+1$
    \item $T_1,...,T_{k-1}: k \rightarrow k$
    \item $d_{-}: k+1 \rightarrow k.$
 \end{itemize}
 \end{varwidth}
 \end{center}
 
 The full set of relations for $\mathbb{A}_{q,t}$ are omitted here but can be found in \cite{CM_2015}. In order to match the parameter conventions in Ion and Wu's work \cite{Ion_2022} we will consider $\mathbb{A}_{t,q}$ as opposed to $\mathbb{A}_{q,t}$ formed by simply swapping $q$ and $t$ in the defining relations of $\mathbb{A}_{q,t}$. Here we highlight a few notable relations of $\mathbb{A}_{t,q}$ which will be required later:
  
\begin{center}
\begin{varwidth}{\textwidth}
\begin{itemize}
    \item The loops $T_1,...,T_{k-1}$ at vertex $k\geq 2$ generate a type $A$ finite Hecke algebra 
    \item $d_{-}^2T_{k-1} = d_{-}^2$ starting at vertex $k \geq 2$
    \item  $T_id_{-} = d_{-}T_i$ at vertex k for $1\leq i\leq k-2$
    \item $z_id_{-}= d_{-}z_i$ at vertex k for $1\leq i \leq k-1$ where $z_1 := \frac{t^k}{1-t}[d_{+}^{*},d_{-}]T_{k-1}^{-1}\cdots T_{1}^{-1}$ and $z_{i+1} = t^{-1}T_iz_iT_i.$
 \end{itemize}
 \end{varwidth}
 \end{center}

\subsubsection{The Standard $\mathbb{A}_{t,q}$ Representation and the +Stable-Limit DAHA}

Vital to the proof of the Compositional Shuffle Conjecture by Carlsson and Mellit \cite{CM_2015} is their construction of a particular representation of $\mathbb{A}_{t,q}$.

\begin{defn}\cite{CM_2015}\label{A qt defns}
    For $k \geq 0$ let $V_k = \mathbb{Q}(q,t)[y_1,\dots,y_k]\otimes \Lambda$ be associated to the vertex $k$ and denote by $V_{\bullet}$ be the system of spaces $V_k$. Let $\zeta_k$ denote the algebra homomorphism $$\zeta_kf(y_1,\dots,.y_{k-1},y_k) = f(y_2,\dots,y_k,qy_1).$$ If $f$ is a formal series with respect to the variable $y$ with coefficients in some ring R denote by $\mathfrak{c}_{y}(f) \in R$ the constant term of $f$ i.e. the coefficient of $y^{0}$ in $f$. Note that each $\mathfrak{S}_k$ acts on $V_k$ by permuting the variables $y_1,...,y_k.$ Define the following operators:
\begin{itemize}
    \item $T_iF = s_iF + (1-t)y_i \frac{F-s_iF}{y_i-y_{i+1}}$
    \item $d_{-}F = \mathfrak{c}_{y_k}(F[X-(t-1)y_k]Exp[-y_k^{-1}X])$
    \item $d_{+}F = -T_1\cdots T_k (y_{k+1}F[X+(t-1)y_{k+1}])$
    \item $d_{+}^{*}F = \zeta_kF[X+(t-1)y_{k+1}].$
\end{itemize}
\end{defn}

\begin{thm}\cite{CM_2015}
    The above operators define a representation of $\mathbb{A}_{t,q}$ on $V_{\bullet}.$
\end{thm}

Ion and Wu use their construction of the standard $\mathscr{H}^{+}$ representation $\sP_{as}^{+}$ to recover the standard $\mathbb{A}_{t,q}$ representation $V_{\bullet}.$

\begin{thm}\cite{Ion_2022}
  There exists an $\mathbb{A}_{t,q}$ representation structure on $\sP_{\bullet} = (\sP(k)^{+})_{k\geq 0}$ isomorphic to the standard representation $V_{\bullet}$ such that at each vertex $k$, $z_i$ acts by $\y_i$ and $y_i$ acts by $X_i$. Further, according to this isomorphism $\sP(k)^{+}$ is identified with $V_k$ via the map $x_1^{a_1}\cdots x_k^{a_k}F[x_{k+1}+\ldots] \rightarrow y_1^{a_1}\cdots y_k^{a_k}F[\frac{X}{t-1}].$
\end{thm} 

\section{Stable-Limits of Non-symmetric Macdonald Polynomials}\label{stable limits of macD poly}

We start by investigating the properties of certain sequences of non-symmetric Macdonald polynomials. We will find that if we fix any composition $\mu$ and consider the sequence of compositions $(\mu*0^m)_{m\geq 0}$ the corresponding sequence of non-symmetric Macdonald polynomials $(E_{\mu*0^m})_{m\geq 0}$ will converge in the sense of Definition \ref{defn6}. It is important to note that in most cases the sequence $(E_{\mu*0^m})_{m\geq 0}$ will not converge with respect to the inverse system $(\pi_{k}: \sP_{k+1} \rightarrow \sP_{k})_{k\geq 1}$. This should be expected because the spectra of the Cherednik operators acting on $\sP_{k+1}$ are incompatible with the spectra from the Cherednik operators acting on $\sP_{k}$. However, by using the HHL explicit combinatorial formula for the non-symmetric Macdonald polynomials we show that the combinatorics of non-attacking labellings underlying the sequence $(E_{\mu*0^m})_{m\geq 0}$ converge in a certain sense. The weaker convergence notion introduced by Ion and Wu is consistent with these combinatorics. For our purposes later in this paper we will heavily rely on the convergence of these sequences as a bridge between the limit Cherednik operators $\y_i$ and their classical counterparts.
 
We now show the convergence of the sequence $(E_{\mu*0^m})_{m \geq 0}$. First, we describe a convenient rearrangement of the monomials in each $E_{\mu*0^m}$.

\begin{thm} \label{HHL tail expansion}

Let $\mu$ be a composition with $\ell(\mu) = n$ and $m \geq 0$. Then $E_{\mu * 0^m}$ has the explicit expression given by 
$$ E_{\mu * 0^m} = \sum_{\substack{\lambda ~ \text{partition}  \\ |\lambda| \leq |\mu|}} m_{\lambda}[x_{n+1}+\ldots + x_{n+m}] \sum_{\substack{\sigma:\mu * 0^{\ell(\lambda)}  \rightarrow [n+\ell(\lambda)]\\ \text{non-attacking} \\ \forall i = 1,...,\ell(\lambda) \\ \lambda_i = |\sigma^{-1}(n+i)|}} x_1 ^{|\sigma^{-1}(1)|}\cdots x_n ^{|\sigma^{-1}(n)|} \Gamma^{(m)} (\widehat{\sigma})$$

where 

$$ \Gamma^{(m)}(\widehat{\sigma}) := q^{-\maj(\widehat{\sigma})}t^{\coinv(\widehat{\sigma})} \prod_{\substack{ u \in dg'(\mu * 0^{\ell(\lambda)}) \\ \widehat{\sigma}(u) \neq \widehat{\sigma}(d(u)) \\ u ~ \text{not in row } 1 }} \left( \frac{1-t}{1-q^{-(\mathleg(u)+1)}t^{(a(u)+1)}} \right) \prod_{\substack{ u \in dg'(\mu * 0^{\ell(\lambda)}) \\ \widehat{\sigma}(u) \neq \widehat{\sigma}(d(u)) \\ u ~ \text{in row } 1 }} \left( \frac{1-t}{1-q^{-(\mathleg(u) +1)}t^{(a(u) + m + 1)}} \right).    \ $$
 
 \end{thm}
 
 \begin{proof}
 First, start with directly applying the HHL formula (\ref{HHL}):
  $$ E_{\mu * 0^m} = \sum_{\substack{\sigma: \mu * 0^m \rightarrow [n + m]\\ \text{non-attacking}}} x^{\sigma}q^{-\maj(\widehat{\sigma})}t^{\coinv(\widehat{\sigma})} \prod_{\substack{u \in dg'(\mu * 0^m) \\ \widehat{\sigma}(u) \neq \widehat{\sigma}(d(u))}} \left( \frac{1-t}{1-q^{-(\mathleg(u)+1)}t^{(a(u)+1)}} \right) .$$
  
  We know that $E_{\mu * 0^m}$ is symmetric in the variables $x_{n+1},...,x_{n+m}$ \cite{C_2001} so it follows that the $\mathbb{Q}(q,t)[x_1,...,x_n]$-coefficient of each monomial in $x_{n+1},...,x_{n+m}$ is independent of the ordering of the latter variables. Hence, we find that by grouping these monomials by symmetry 

\begin{multline*}
    E_{\mu * 0^m} = \sum_{\lambda} m_{\lambda}[x_{n+1}+...+x_{n+m}] \sum_{\substack{\sigma: \mu * 0^m \rightarrow [n + m] \\ \text{non-attacking} \\ \forall i ~ \lambda_i = |\sigma^{-1}(n+i)|  }} x_1 ^{|\sigma^{-1}(1)|}\cdots x_n ^{|\sigma^{-1}(n)|} q^{-\maj(\widehat{\sigma})}t^{\coinv(\widehat{\sigma})} ~ \times
    \\ \prod_{\substack{u \in dg'(\mu) \\ \widehat{\sigma}(u) \neq \widehat{\sigma}(d(u))}} \left( \frac{1-t}{1-q^{-(\mathleg(u)+1)}t^{(a(u)+1)}} \right).
\end{multline*}
 
 Note that by degree considerations the only possible partitions $\lambda$ that have a nonzero contribution to the above sum have $|\lambda| \leq |\mu|$ and hence we can rewrite the above sums as 
  $$ \sum_{\lambda} \sum_{\substack{\sigma: \mu * 0^m \rightarrow [n + m] \\ \text{non-attacking} \\ \forall i ~ \lambda_i = |\sigma^{-1}(n+i)|  }} = \sum_{\substack{\lambda ~ \text{partition}  \\ |\lambda| \leq |\mu|}} \sum_{\substack{\sigma:\mu* 0^m \rightarrow [n+\ell(\lambda)]\\ \text{non-attacking} \\ \forall i ~ \lambda_i = |\sigma^{-1}(n+i)|}}    .$$
In the latter sum above we have written each $\sigma$ as a non-attacking labelling $\sigma: \mu*0^m \rightarrow [n+\ell(\lambda)]$ to emphasize that the numbers occurring in this labelling are contained in the set $[n+\ell(\lambda)]$ which is independent of $m$. However, these are still considered labellings of the diagram corresponding to $\mu*0^m$ and hence we calculate the corresponding $q,t$ coefficients in the HHL formula accordingly.
  
  We must now understand the dependence on $m$ of the statistics $\maj$, $\coinv$, $\mathleg$, and $a$ in each of the non-attacking labellings $\sigma: \mu*0^m \rightarrow [n+\ell(\lambda)]$ as $m$ varies. Fix a non-attacking labelling $\sigma: \mu * 0^k \rightarrow [n+k]$ for some $k \leq m$ and let $\sigma_m$ be the associated labelling of $\mu * 0^m$. Recall that 
  $$\maj(\widehat{\sigma}) = \sum_{u \in \Des(\widehat{\sigma})} (\mathleg(u) + 1)$$
  and similarly for $\maj(\widehat{\sigma}_m).$
  The only descent boxes of $\widehat{\sigma}_m$ occur in the diagram $dg'(\mu)$ itself and $\mathleg(u)$ for these boxes will not depend on m. Therefore, $\maj(\widehat{\sigma}_m) = \maj(\widehat{\sigma})$. For $u \in dg'(\mu * 0^m)$ clearly $u \in dg'(\mu)$ and by direct computation we see that when $u$ is not in row 1 then $a(u)$ does not depend on m. However, for u in row 1 $a(u)$ when calculated in the diagram $\widehat{dg}(\mu)$ increases to $a(u)+m$ when calculated in the diagram $\widehat{dg}(\mu*0^m)$. This comes from counting the extra row 0 boxes for each box in row 1. Also note that in any non-attacking labelling there cannot be descent boxes in row 1. Now from careful counting we get the following:
  \begin{itemize}
      \item $|\Inv(\widehat{\sigma}_m)| = |\Inv(\widehat{\sigma})| + (n+k)(m-k) + \binom{m-k}{2}$
      \item $ |\{i<j : (\mu * 0^m)_i \leq (\mu * 0^m)_j \} | = | \{i<j : (\mu*0^k)_i \leq (\mu*0^k)_j \} | + (\#\{i:\mu_i = 0\} +k)(m-k) + \binom{m-k}{2}$
      \item $\sum_{u \in \Des(\widehat{\sigma}_m)} a(u) = \sum_{u \in \Des(\widehat{\sigma})} a(u).$
  \end{itemize}

By using the above calculations and cancelling out terms we get

\begin{align*}
    \inv(\widehat{\sigma}_m) &= |\Inv(\widehat{\sigma}_m)| - | \{i<j: (\mu * 0^m)_i \leq (\mu * 0^m)_j \} | - \sum_{u \in \Des(\widehat{\sigma}_m)} a(u) \\
    &= |\Inv(\widehat{\sigma})| - | \{i<j: (\mu*0^k)_i \leq (\mu*0^k)_j \} | -\sum_{u \in \Des(\widehat{\sigma})} a(u) + (n- \#\{i:\mu_i = 0\})(m-k)\\
    &= \inv(\widehat{\sigma}) +\#\{i:\mu_i \neq 0\}(m-k).\\
\end{align*} 

  Further, from the prior observation about how arm, $a(u)$, changes with m we see that 
  $$\sum_{u \in dg'(\mu * 0^m)} a(u) = \#\{i:\mu_i \neq 0\}(m-k) + \sum_{u \in dg'(\mu*0^k)} a(u)$$
  where arm has been calculated in the corresponding diagrams. 

We then have 
\begin{align*}
    \coinv(\widehat{\sigma}_m)&= \left( \sum_{u \in dg'(\mu * 0^m)} a(u) \right) - \inv(\widehat{\sigma}_m) \\
    &= \left(\#\{i:\mu_i \neq 0\}(m-k) + \sum_{u \in dg'(\mu*0^k)} a(u)\right) - \left( \inv(\widehat{\sigma}) + \#\{i:\mu_i \neq 0\}(m-k)\right)\\
    & = \left( \sum_{u \in dg'(\mu*0^k)} a(u)\right) - \inv(\widehat{\sigma}) \\
    &= \coinv(\widehat{\sigma}).\\
\end{align*}
  Thus $\maj(\widehat{\sigma}_m) = \maj(\widehat{\sigma})$ and $\coinv(\widehat{\sigma}_m)= \coinv(\widehat{\sigma}).$

Lastly, we return to the expansion of $E_{\mu*0^m}$ we found above. For each partition $\lambda$ with $|\lambda| \leq |\mu|$ we now see that 

\begin{align*}
     &\sum_{\substack{\sigma:\mu* 0^m \rightarrow [n+\ell(\lambda)]\\ \text{non-attacking} \\ \forall i ~ \lambda_i = |\sigma^{-1}(n+i)|}} x_1 ^{|\sigma^{-1}(1)|}\cdots x_n ^{|\sigma^{-1}(n)|} q^{-\maj(\widehat{\sigma})}t^{\coinv(\widehat{\sigma})} \prod_{\substack{u \in dg'(\mu) \\ \widehat{\sigma}(u) \neq \widehat{\sigma}(d(u))}} \left( \frac{1-t}{1-q^{-(\mathleg(u)+1)}t^{(a(u)+1)}} \right) \\
     &= \sum_{\substack{\sigma:\mu* 0^{\ell(\lambda)} \rightarrow [n+\ell(\lambda)]\\ \text{non-attacking} \\ \forall i ~ \lambda_i = |\sigma^{-1}(n+i)|}} x_1 ^{|\sigma^{-1}(1)|}\cdots x_n ^{|\sigma^{-1}(n)|} \Gamma^{(m)}(\widehat{\sigma}).
\end{align*}

where
$$ \Gamma^{(m)}(\widehat{\sigma}) := q^{-\maj(\widehat{\sigma})}t^{\coinv(\widehat{\sigma})} \prod_{\substack{ u \in dg'(\mu * 0^{\ell(\lambda)}) \\ \widehat{\sigma}(u) \neq \widehat{\sigma}(d(u)) \\ u ~ \text{not in row } 1 }} \left( \frac{1-t}{1-q^{-(\mathleg(u)+1)}t^{(a(u)+1)}} \right) \prod_{\substack{ u \in dg'(\mu * 0^{\ell(\lambda)}) \\ \widehat{\sigma}(u) \neq \widehat{\sigma}(d(u)) \\ u ~ \text{in row } 1 }} \left( \frac{1-t}{1-q^{-(\mathleg(u) +1)}t^{(a(u) + m + 1)}} \right).    \ $$

and we calculate all of the associated statistics in their respective diagrams.

 \end{proof}

Now that we have conveniently rearranged the monomial terms of each $E_{\mu*0^m}$ and identified the dependence of the coefficients on the parameter $m$ we can give a simple proof that the sequence $(E_{\mu*0^m})_{m \geq 0}$ converges.

\begin{cor/defn}\label{convergence of macdonald}

Let $\mu$ be a composition with $\ell(\mu) = n$. The sequence $(E_{\mu * 0^m})_{m \geq 1}$ converges to an almost-symmetric function $\widetilde{E}_{\mu}:= \lim_{m} E_{\mu * 0^m} \in \sP_{as}^{+}$ given explicitly by 
$$ \widetilde{E}_{\mu} = \sum_{\substack{\lambda ~ \text{partition}  \\ |\lambda| \leq |\mu|}} m_{\lambda}[x_{n+1}+\ldots] \sum_{\substack{\sigma:\mu * 0^{\ell(\lambda)} \rightarrow [n+\ell(\lambda)]\\ \text{non-attacking} \\ \forall i = 1,...,\ell(\lambda) \\ \lambda_i = |\sigma^{-1}(n+i)|}} x_1 ^{|\sigma^{-1}(1)|}\cdots x_n ^{|\sigma^{-1}(n)|} \widetilde{\Gamma}(\widehat{\sigma})$$  

where 

$$\widetilde{\Gamma}(\widehat{\sigma}) := \lim_{m} \Gamma^{(m)}(\widehat{\sigma}) = q^{-\maj(\widehat{\sigma})}t^{\coinv(\widehat{\sigma})} \prod_{\substack{ u \in dg'(\mu * 0^{\ell(\lambda)}) \\ \widehat{\sigma}(u) \neq \widehat{\sigma}(d(u)) \\ u ~ \text{not in row } 1 }} \left( \frac{1-t}{1-q^{-(\mathleg(u)+1)}t^{(a(u)+1)}} \right) \prod_{\substack{ u \in dg'(\mu * 0^{\ell(\lambda)}) \\ \widehat{\sigma}(u) \neq \widehat{\sigma}(d(u)) \\ u ~ \text{in row } 1 }} \left( 1-t \right). $$

\end{cor/defn}

\begin{proof}
Note that the formula in Theorem \ref{HHL tail expansion} is a fixed size finite sum where the only dependence on m is in the $m_{\lambda}$ symmetric function terms and the $t^m$ occurring in the $\Gamma^{(m)}$ terms. Thus in the sense of Ion and Wu, see Definition \ref{defn6}, this sequence converges to a well defined element of $\sP_{as}^{+}$. In particular, each $m_{\lambda}[x_{n+1}+\ldots +x_{n+m}]$ converges to $m_{\lambda}[x_{n+1}+\ldots]$ and $t^m$ converges to 0 in the $\widetilde{\Gamma}$-term. Simplifying gives the formula above. 

\end{proof}

It follows from Corollary \ref{convergence of macdonald} that the almost symmetric functions $\widetilde{E}_{\mu}$ are homogeneous of degree $|\mu|$ and $\widetilde{E}_{\mu} \in \sP(\ell(\mu))^{+}.$ Note importantly, that for any composition $\mu$ (not necessarily reduced) and any $n \geq 0$, by shifting the terms of the sequence $(E_{\mu*0^m})_{m\geq 0}$ we see that $\widetilde{E}_{\mu *0^n} = \widetilde{E}_{\mu}$.

\begin{cor}\label{stabilizing of convergence of macdonald for partitions}
Let $\lambda$ be a partition with $\ell(\lambda) = n$ and $|\lambda| = N$. Then $\widetilde{E}_{\lambda}$ is determined by  $E_{\lambda*0^N} \in \sP_{n+N}^{+}$. That is to say, if 
$$E_{\lambda*0^{N}}(x_1,...,x_{n+N}) = c_1 x^{\mu^{(1)}}m_{\nu^{(1)}}[x_{n+1}+...+x_{n+N}] +...+ c_k x^{\mu^{(k)}}m_{\nu^{(k)}}[x_{n+1}+...+x_{n+N}]$$ then 
$$ \widetilde{E}_{\lambda} = c_1 x^{\mu^{(1)}}m_{\nu^{(1)}}[x_{n+1}+...] +...+ c_k x^{\mu^{(k)}}m_{\nu^{(k)}}[x_{n+1}+...] .$$
\end{cor}

\begin{proof}
As $\lambda$ is a partition, row 1 of any non-attacking labelling of $\lambda$ must be 1,2,,...,$\ell(\lambda)$. Thus no boxes of $dg'(\lambda)$ in row 1 will have $\widehat{\sigma}(u) \neq \widehat{\sigma}(d(u))$ and so there will be no contributions from any of the terms of the form 
$$\prod_{\substack{ u \in dg'(\lambda) \\ \widehat{\sigma}(u) \neq \widehat{\sigma}(d(u)) \\ u ~ \text{row } 1 }} \left( \frac{1-t}{1-q^{-(\mathleg(u) +1)}t^{(a(u) + m + 1)}} \right).$$ Further, from  Corollary \ref{convergence of macdonald} it is clear that these are the only coefficients that depend on m in the limit. Also it follows that each term of the form $x^{\mu}m_{\nu}[x_{n+1}+\ldots]$ that occurs in the expansion of $\widetilde{E}_{\lambda}$ appears at least by the $m = N$ step of the limit. From these two facts it follows that the expansion of $\widetilde{E}_{\lambda}$ will match that of $E_{\lambda*0^{N}}(x_1,...,x_{n+N})$ up to truncating each $m_{\nu}[x_{n+1}+\ldots]$ to $m_{\nu}[x_{n+1}+\ldots + x_{n+N}]$ using $\pi_{n+N}.$

\end{proof}

\section{$\mathscr{Y}$-Weight Basis of $\mathscr{P}_{as}^{+}$}\label{weight basis section}

Given a family of commuting operators $\{ y_i : i \in I \}$ and a weight vector v we denote its weight by the function $\alpha: I \rightarrow \mathbb{Q}(q,t)$ such that $y_iv = \alpha(i)v.$ We sometimes denote $\alpha$ as $(\alpha_1,\alpha_2,\ldots).$

\subsection{The $\widetilde{E}_{\mu}$ are $\mathscr{Y}$-Weight Vectors}

In what follows, the classical spectral theory for non-symmetric Macdonald polynomials is used to demonstrate that the limit functions $\widetilde{E}_{\mu}$ are $\mathscr{Y}$-weight vectors. The below lemma is a simple application of this classical theory and basic properties of the $t$-adic topology on $\mathbb{Q}(q,t)$.

\begin{lem} \label{convergence of eigenvalues}
For a composition $\mu$ with $\ell(\mu) = n$ define $\alpha_{\mu}^{(m)}$ to be the $Y^{(n+m)}$-weight of $E_{\mu *0^m}$. Then in the $t$-adic topology on $\mathbb{Q}(q,t)$ the sequence $(t^{n+m}\alpha_{\mu}^{(m)}(i))_{m\geq 0}$ converges in m to some $\widetilde{\alpha}_{\mu}(i) \in \mathbb{Q}(q,t)$. In particular, $\widetilde{\alpha}_{\mu}(i) = 0$ for $i > n$ and for $1\leq i\leq n$ we have that $\widetilde{\alpha}_{\mu}(i) = 0$  exactly when $\mu_i = 0$.
\end{lem}

\begin{proof}
Take $\mu = (\mu_1,\ldots,\mu_n)$. From classical double affine Hecke algebra theory \cite{C_2001}  we have $\alpha_{\mu}^{(0)}(i) = q^{\mu_i}t^{1-\beta_{\mu}(i)}$ where 
$$\beta_{\mu}(i) := \#\{j: 1\leq j \leq i ~, \mu_j \leq \mu_i\} + \#\{j: i < j \leq n ~, \mu_i > \mu_j\}.$$
If we calculate $\beta_{\mu*0^m}(i)$ directly it follows then that

\[ t^{n+m}\alpha_{\mu}^{(m)}(i) = 
    \begin{cases}
    q^{\mu_i}t^{n+m+1-(\beta_{\mu}(i) + m \mathbbm{1}(\mu_i \neq 0))} = t^n\alpha_{\mu}^{(0)}(i) & i\leq n, \mu_i \neq 0\\
    q^{\mu_i}t^{n+m+1-(\beta_{\mu}(i) + m \mathbbm{1}(\mu_i \neq 0))} = t^{n+m}\alpha_{\mu}^{(0)}(i) & i\leq n, \mu_i = 0 \\
    t^{n+m+1-(\#\{j:\mu_{j} = ~0\} + i-n)} = t^{\#\{j:\mu_j \neq ~0\}}t^{m+1-(i-n)} & i > n .
     \end{cases}
\]

Lastly, by taking the limit $m\rightarrow \infty$ we get the result.
\end{proof}

For a composition $\mu$ define the weight $\widetilde{\alpha}_{\mu}$ using the formula in Lemma \ref{convergence of eigenvalues} for the list of scalars  $\widetilde{\alpha}_{\mu}(i)$ for $i \in \mathbb{N}$.

\begin{lem} \label{weight vector if nonzero parts}
For a composition $\mu = (\mu_1,\ldots,\mu_n)$ with $\mu_i \neq 0$ for $1\leq i \leq n$, $\widetilde{E}_{\mu}$ is a $\y$-weight vector with weight $\widetilde{\alpha}_{\mu}$. 
\end{lem}

\begin{proof}
Fix any $r \in \mathbb{N}$. We start by rewriting the operator $\y_r$ explicitly in terms of the limit definition of $\y_1$.

\begin{align*}
\y_r &= t^{-(r-1)}T_{r-1}\cdots T_1 \y_1 T_1 \cdots T_{r-1}\\
&= t^{-(r-1)}T_{r-1}\cdots T_1 \lim_{k} t^{k} \rho \omega_k ^{-1} T_{k-1}^{-1}\cdots T_{1}^{-1}T_1\cdots T_{r-1} \pi_k \\
&= \lim_{k} t^{k} T_{r-1}\cdots T_1 \rho t^{-(r-1)} \omega_k ^{-1} T_{k-1}^{-1}\cdots T_{r}^{-1}\pi_k \\
&= \lim_{k} t^{k} T_{r-1}\cdots T_1 \rho T_1^{-1} \cdots T_{r-1}^{-1} t^{-(r-1)} T_{r-1} \cdots T_{1}  \omega_k ^{-1} T_{k-1}^{-1}\cdots T_{r}^{-1} \pi_k \\
&= \lim_{k} t^{k} T_{r-1}\cdots T_1 \rho T_1^{-1} \cdots T_{r-1}^{-1} Y_{r}^{(k)}\pi_k. \\
\end{align*}

Applying $\y_r$ to $\widetilde{E}_{\mu}$ we see by taking $k = n+m \geq n$ and shifting the indices that

\begin{align*}
\y_r(\widetilde{E}_{\mu}) &= \lim_{m} t^{n+m} T_{r-1}\cdots T_1 \rho T_1^{-1} \cdots T_{r-1}^{-1} Y_{r}^{(n+m)}(E_{\mu*0^{m}}) \\
&= \lim_{m} T_{r-1}\cdots T_1 \rho T_1^{-1} \cdots T_{r-1}^{-1}  t^{n+m} \alpha_{\mu}^{(m)}(r) E_{\mu*0^{m}} \\
\end{align*}
and by Lemma \ref{convergence of eigenvalues} this converges to 
$$ \y_r( \widetilde{E}_{\mu}) = \widetilde{\alpha}_{\mu}(r) (T_{r-1}\cdots T_1 \rho T_1^{-1}\cdots T_{r-1}^{-1}) \widetilde{E}_{\mu} .$$ 

Importantly, we have implicitly used the fact that both of the sequences $(E_{\mu*0^{m}})_{m}$ and $(\alpha_{\mu}^{(m)}(r))_{m}$ converge, that the operator $T_{r-1}\cdots T_1 \rho T_1^{-1}\cdots T_{r-1}^{-1}$ commutes with the quotient maps $\pi_k: \sP_{k+1} \rightarrow \sP_{k}$ for $k> r$, and Proposition 6.21 in \cite{Ion_2022}.
We will show that the right side is $\widetilde{\alpha}_{\mu}(r)\widetilde{E}_{\mu}$.
As $\widetilde{\alpha}_{\mu}(r) = 0$ for $r > n$ by Lemma \ref{convergence of eigenvalues} we reduce to the sub case $r \leq n$. Fix $r \leq n$. 
If we could show that $x_1$ divides $ T_1^{-1}\cdots T_{r-1}^{-1}\widetilde{E}_{\mu}$ then we would have $$\rho(T_1^{-1}\cdots T_{r-1}^{-1}\widetilde{E}_{\mu}) = T_1^{-1}\cdots T_{r-1}^{-1}\widetilde{E}_{\mu} $$
implying that 
\begin{align*}
\y_r(\widetilde{E}_{\mu}) &= \widetilde{\alpha}_{\mu}(r) (T_{r-1}\cdots T_1 \rho T_1^{-1}\cdots T_{r-1}^{-1}) \widetilde{E}_{\mu} \\
&= \widetilde{\alpha}_{\mu}(r) T_{r-1}\cdots T_1 T_1^{-1}\cdots T_{r-1}^{-1}) \widetilde{E}_{\mu} \\
&= \widetilde{\alpha}_{\mu}(r)\widetilde{E}_{\mu}
\end{align*}
as desired.
To show that $x_1 | T_1^{-1}\cdots T_{r-1}^{-1}\widetilde{E}_{\mu}$ it suffices to show that for all $m \geq 0$,
$x_1 | T_1^{-1}\cdots T_{r-1}^{-1} E_{\mu *0^m}$. To this end fix $m \geq 0$. We have that 

\begin{align*}
\alpha_{\mu}^{(m)}(r)E_{\mu*0^m}&= Y_r^{(n+m)}(E_{\mu*0^m}) \\
&= t^{n+m-r+1}T_{r-1}\cdots T_1 \omega_{n+m}^{-1} T_{n+m-1}^{-1}\cdots T_r^{-1} E_{\mu*0^m}.\\
\end{align*}

Since $\alpha_{\mu}^{(m)}(r) \neq 0$ we can have $\frac{1}{\alpha_{\mu}^{(m)}(r)} T_1 ^{-1}\cdots T_{r-1}^{-1}$ act on both sides of the above to get 

\begin{align*}
T_1^{-1}\cdots T_{r-1}^{-1} E_{\mu*0^m} &= \frac{t^{n+m-r+1}}{\alpha_{\mu}^{(m)}(r)}\omega_{n+m}^{-1} T_{n+m-1}^{-1}\cdots T_r^{-1} E_{\mu*0^m}.
\end{align*}
 By HHL any non-attacking labelling of $\mu* 0^m$ will have row 1 diagram labels given by $\{1,2,\ldots,n\}$ so in particular $x_r$ divides $E_{\mu*0^m}$ for all $m > 0$. Lastly, 
 \begin{align*}
\omega_{n+m}^{-1} T_{n+m-1}^{-1}\cdots T_r^{-1}X_r &= \omega_{n+m}^{-1} t^{-(n+m-r)}X_{n+m}T_{n+m-1}\cdots T_r \\
&= qt^{-(n+m-r)} X_1 \omega_{n+m}^{-1} T_{n+m-1}\cdots T_r.\\
 \end{align*}
Thus $x_1$ divides $T_1^{-1}\cdots T_{r-1}^{-1} E_{\mu*0^m}$ for all $m \geq 0$ showing the result.

\end{proof}

Now we consider the general situation where the composition $\mu$ can have some parts which are 0. We can extend the above result, Lemma \ref{weight vector if nonzero parts}, by a straight-forward argument using intertwiner theory from the study of affine Hecke algebras.

\begin{thm} \label{weight vector for all}
For all compositions $\mu$, $\widetilde{E}_{\mu}$ is a $\y$-weight vector with weight $\widetilde{\alpha}_{\mu}.$
\end{thm}

\begin{proof}
 Lemma \ref{weight vector if nonzero parts} shows that this statement holds for any composition with all parts nonzero. Fix a composition $\mu$ with length $n$. We know that by sorting in decreasing order that $\mu$ can be written as a permutation of a composition of the form $\nu * 0^m$ for a partition $\nu$ and some $m \geq 0$. From the definition of Bruhat order it follows that $\nu*0^m$ will be the minimal element out of all of its distinct permutations, including $\mu$. Necessarily, this finite subposet generated by the permutations of $\nu*0^m$ is isomorphic to the Bruhat ordering on the coset space $\mathfrak{S}_n/\mathfrak{S}_{\kappa}$ where $\mathfrak{S}_{\kappa}$ is the Young subgroup of $\mathfrak{S}_n$ corresponding to the stabilizer of $\nu*0^m$. Hence, it suffices to show inductively that for any composition $\beta$ with $\nu*0^m \leq \beta < s_i(\beta) \leq \mu$, if $\widetilde{E}_{\beta}$ satisfies the theorem then so will $\widetilde{E}_{s_i(\beta)}$. As $\mu$ is finitely many covering elements away in Bruhat from $\nu*0^m$ this induction will indeed terminate after finitely many steps.
 
 Using the intertwiner operators from affine Hecke algebra theory, given by $\varphi_i = T_i\y_i-\y_iT_i$ in this context, we only need to show that for any composition $\beta$ with $\nu*0^m \leq \beta < s_i(\beta) \leq \mu$, $$\varphi_i \widetilde{E}_{\beta} = (\widetilde{\alpha}_{\beta}(i) - \widetilde{\alpha}_{\beta}(i+1))\widetilde{E}_{s_i(\beta)}.$$ Suppose the theorem holds for some $\beta$ with $\nu*0^m \leq \beta < s_i(\beta) \leq \mu$. Then we have the following:
\begin{align*}
    \varphi_{i}\widetilde{E}_{\beta} &= (T_i(\y_i - \y_{i+1}) + (1-t)\y_{i+1})\widetilde{E}_{\beta} \\
    &= (\widetilde{\alpha}_{\beta}(i)- \widetilde{\alpha}_{\beta}(i+1))T_i\widetilde{E}_{\beta} + (1-t)\widetilde{\alpha}_{\beta}(i+1)\widetilde{E}_{\beta} \\
    &= \lim_{m} (t^{n+m}\alpha_{\beta}^{(m)}(i)- t^{n+m}\alpha_{\beta}^{(m)}(i+1))T_iE_{\beta *0^m} + (1-t)t^{n+m}\alpha_{\beta}^{(m)}(i+1)E_{\beta *0^m} \\
    &= \lim_{m}(t^{n+m}\alpha_{\beta}^{(m)}(i)- t^{n+m}\alpha_{\beta}^{(m)}(i+1))E_{s_i(\beta) *0^m}\\
    &= (\widetilde{\alpha}_{\beta}(i) - \widetilde{\alpha}_{\beta}(i+1))\widetilde{E}_{s_i(\beta)}.\\
\end{align*}

\end{proof}

As an immediate consequence of the proof of Theorem \ref{weight vector for all} we have the following. 

\begin{cor}\label{action of T on E mu}
    Let $\mu$ be a composition and $i \geq 1$ such that $s_i(\mu) > \mu$. Then 
    $$\widetilde{E}_{s_i(\mu)} = \left(T_i + \frac{(1-t)\widetilde{\alpha}_{\mu}(i+1)}{\widetilde{\alpha}_{\mu}(i) - \widetilde{\alpha}_{\mu}(i+1)}\right) \widetilde{E}_{\mu}.$$
\end{cor}

We have shown in Theorem \ref{weight vector for all} there is an explicit collection of $\mathscr{Y}$-weight vectors $\widetilde{E}_{\mu}$ in $\mathscr{P}_{as}^{+}$ arising as the limits of non-symmetric Macdonald polynomials $E_{\mu * 0^m}$. Unfortunately, these $\widetilde{E}_{\mu}$ do not span $\mathscr{P}_{as}^{+}$. To see this note that one cannot write a non-constant symmetric function as a linear combination of the $\widetilde{E}_{\mu}$. However, in the below work we build a full $\y$-weight basis of $\sP_{as}^{+}$. 

\subsection{Constructing a Full $\y$-Weight Basis}

\subsubsection{Defining the Stable-Limit Non-symmetric Macdonald Functions}
 To complete our construction of a full weight basis of $\mathscr{P}_{as}^{+}$ we will need the $\partial_{-}^{(k)}$ operators from Ion and Wu. These operators are, up to a change of variables and plethsym, the $d_{-}$ operators from Carlsson and Mellit's standard $\mathbb{A}_{t,q}$ representation. 

\begin{defn}\cite{Ion_2022} \label{defn7}
 Define the operator $\partial_{-}^{(k)}: \mathscr{P}(k)^{+} \rightarrow \mathscr{P}(k-1)^{+}$ to be the $\mathscr{P}_{k-1}^{+}$-linear map which acts on elements of the form $x_k^nF[x_{k+1}+x_{k+2}+\ldots]$ for $F \in \Lambda$ and $n \geq 0$ as 
 $$\partial_{-}^{(k)}(x_k^nF[x_{k+1}+x_{k+2}+\ldots]) = \mathscr{B}_n(F)[x_k + x_{k+1} +\ldots]. $$
 Here the $\mathscr{B}_n$ are the Jing operators which serve as creation operators for Hall-Littlewood symmetric functions $\mathcal{P}_{\lambda}$ given explicitly by the following plethystic formula:
 $$\mathscr{B}_n(F)[X] = \langle z^n \rangle  F[X-z^{-1}]Exp[(1-t)zX] .$$
 
\end{defn}
 Importantly, the $\partial_{-}^{(k)}$ operators do not come from the $\mathscr{H}^{+}$ action itself. Note that the $\partial_{-}^{(k)}$ operators are homogeneous by construction. 

We will require the useful alternative expression for the $\partial_{-}^{(k)}$ operators which can be found in \cite{Ion_2022}. Recall the notation $\mathfrak{c}_{y}$ from Definition \ref{A qt defns}.

\begin{lem}\label{lowering operator alternative description}
Let $\tau_k$ denote the alphabet shift $\mathfrak{X}_k \rightarrow \mathfrak{X}_{k-1}$ acting on symmetric functions where $\mathfrak{X}_i : = x_{i+1}+x_{i+2}+...$. Then for $f \in \sP_{k}$ and $F \in \Lambda$
$$ \partial_{-}^{(k)}(f(x_1,\dots x_k) F[\mathfrak{X}_k]) = \tau_k \mathfrak{c}_{x_k} f(x_1,...,x_k)F[\mathfrak{X}_k -x_k]Exp[-(t-1)x_k^{-1}\mathfrak{X}_{k}].$$
\end{lem}

\begin{proof}
\cite{Ion_2022}.
\end{proof}

As an immediate consequence of this explicit description of the action of the $\partial_{-}^{(k)}$ operator we get the following required lemmas.

\begin{lem}\cite{Ion_2022} \label{lowering is a projection}
The map $\partial_{-}^{(k)}: \sP(k)^{+} \rightarrow \sP(k-1)^{+}$ is a projection onto $\sP(k-1)^{+}$ i.e. for $f \in \sP(k-1)^{+}\subset \sP(k)^{+}$ we have that $\partial_{-}^{(k)}(f) = f$.
\end{lem}
\begin{proof}
    Fix $F \in \Lambda.$ It suffices to show that $\partial_{-}^{(k)}(F[\mathfrak{X}_{k-1}]) = F[\mathfrak{X}_{k-1}].$ By using the coproduct on $\Lambda$ we can expand $F[\mathfrak{X}_{k-1}] = F[x_k+\mathfrak{X}_{k}]$ in powers of $x_k^{i}$ with some coefficients $F_i \in \Lambda$ as $F[x_k+\mathfrak{X}_{k}] = \sum_{i \geq 0} x_k^{i}F_{i}[\mathfrak{X}_{k}].$ From Lemma \ref{lowering operator alternative description} we have 
    \begin{align*}
        \partial_{-}^{(k)}(F[\mathfrak{X}_{k-1}]) &= \partial_{-}^{(k)}(F[x_{k}+\mathfrak{X}_{k}])\\
        &= \partial_{-}^{(k)}(\sum_{i \geq 0} x_k^{i}F_{i}[\mathfrak{X}_{k}]) \\
        &= \tau_k \mathfrak{c}_{x_k} \left( \sum_{i \geq 0} x_k^{i}F_i[\mathfrak{X}_k -x_k]Exp[-(t-1)x_k^{-1}\mathfrak{X}_{k}] \right) \\
        &= \tau_k \mathfrak{c}_{x_k} F[\mathfrak{X}_k -x_k +x_k]Exp[-(t-1)x_k^{-1}\mathfrak{X}_{k}] \\
        &= \tau_k \mathfrak{c}_{x_k} F[\mathfrak{X}_k] Exp[-(t-1)x_k^{-1}\mathfrak{X}_{k}] \\ 
        &= \tau_k F[\mathfrak{X}_k] \mathfrak{c}_{x_k} Exp[-(t-1)x_k^{-1}\mathfrak{X}_{k}] \\ 
        &= \tau_k F[\mathfrak{X}_k] \\
        &= F[\mathfrak{X}_{k-1}]. \\
    \end{align*}

\end{proof}
\begin{lem}\label{lowering is a partial sym module map}
For all $G \in \Lambda$ and $g(x) \in \sP(k)^{+}$ 
$$ \partial_{-}^{(k)}(G[x_k + x_{k+1}+\dots]g(x)) = G[x_k + x_{k+1}+\dots] \partial_{-}^{(k)}(g(x)).$$
\end{lem}

\begin{proof}
It suffices to take $g(x) \in \sP(k)^{+}$ to be of the form $g(x) = f(x_1,\dots,x_k)F[\mathfrak{X}_{k}]$ with $f\in \sP_{k}^{+}$ and $F \in \Lambda$.
From Lemma \ref{lowering operator alternative description} we get the following:

\begin{align*}
\partial_{-}^{(k)}(G[x_k + x_{k+1}+\dots]g(x)) &= \partial_{-}^{(k)}(G[\mathfrak{X}_{k-1}]g(x)) \\
&= \tau_k \mathfrak{c}_{x_k} G[\mathfrak{X}_{k-1}-x_k] f(x_1,\dots,x_k)F[\mathfrak{X}_{k} - x_k]Exp[-(t-1)x_k^{-1}\mathfrak{X}_{k}] \\
&= \tau_k \mathfrak{c}_{x_k} G[\mathfrak{X}_{k}] f(x_1,\dots,x_k)F[\mathfrak{X}_{k} - x_k]Exp[-(t-1)x_k^{-1}\mathfrak{X}_{k}] \\
&= \tau_kG[\mathfrak{X}_{k}]\mathfrak{c}_{x_k}f(x_1,\dots,x_k)F[\mathfrak{X}_{k} - x_k]Exp[-(t-1)x_k^{-1}\mathfrak{X}_{k}]\\
&= G[\mathfrak{X}_{k-1}]\tau_k\mathfrak{c}_{x_k}f(x_1,\dots,x_k)F[\mathfrak{X}_{k} - x_k]Exp[-(t-1)x_k^{-1}\mathfrak{X}_{k}]\\
&= G[\mathfrak{X}_{k-1}]\partial_{-}^{(k)}(f(x_1,\dots,x_k)F[\mathfrak{X}_{k}]) \\
&= G[\mathfrak{X}_{k-1}]\partial_{-}^{(k)}(g(x)).\\
\end{align*}
    
\end{proof}

\begin{cor}\label{lowering is a sym module map}
    For $G \in \Lambda$ and $g(x) \in \sP(k)^{+}$
    $$\partial_{-}^{(k)}(G[X]g(x)) = G[X] \partial_{-}^{(k)}(g(x)).$$
\end{cor}

\begin{proof}
    Take $G \in \Lambda$ and $g(x) \in \sP(k)^{+}$. Expand $G[X]$ as a finite sum of terms of the form $f_i(x_1,\dots,x_{k-1})F_i[x_k +\ldots]$, where $f_i \in \sP_{k-1}$ and $F_i \in \Lambda$ so 
    $$G[X] = \sum_{i} f_i(x_1,\dots,x_{k-1})F_i[x_k +\ldots].$$  By Lemma \ref{lowering is a partial sym module map} and the fact that $\partial_{-}^{(k)}$ is a $\sP_{k-1}^{+}$-linear map from Definition \ref{defn7} we now see that 

\begin{align*}
    \partial_{-}^{(k)}(G[X]g(x)) &= 
    \sum_{i} \partial_{-}^{(k)}( f_i(x_1,\dots,x_{k-1})F_i[x_k +\ldots]g(x)) \\
    &= \sum_{i} f_i(x_1,\dots,x_{k-1})F_i[x_k +\ldots]\partial_{-}^{(k)}(g(x)) \\
    &= G[X]\partial_{-}^{(k)}(g(x)).\\
\end{align*}

\end{proof}

We can now construct a full $\y$-weight basis of $\mathscr{P}_{as}^{+}$. We parameterize this basis by pairs $(\mu | \lambda)$ for $\mu$ a reduced composition and $\lambda$ a partition. Combinatorially, this is reasonable because, as already mentioned, the monomial basis for $\sP_{as}^{+}$, $\{x^{\mu}m_{\lambda} \mid \mu \in \Compred, \lambda \in \Par\}$, is indexed by pairs of reduced compositions and partitions.

\begin{defn} \label{defn9}
For $\mu$ a reduced composition and $\lambda$ a partition define the \textbf{\textit{stable-limit non-symmetric Macdonald function}} corresponding to $(\mu|\lambda)$ as
$$\widetilde{E}_{(\mu|\lambda)} :=  \partial_{-}^{(\ell(\mu)+1)}\cdots \partial_{-}^{(\ell(\mu)+\ell(
\lambda))} \widetilde{E}_{\mu*\lambda}. $$For a partition $\lambda$ define
\begin{equation}\label{A-Lambda basis}
\mathcal{A}_{\lambda} := \widetilde{E}_{(\emptyset|\lambda)} \in \Lambda .
\end{equation}

\end{defn}

Later in Theorem \ref{main theorem}, we will show that the collection $\{\widetilde{E}_{(\mu|\lambda)} \mid \mu \in \Compred, \lambda \in \Par\}$ is a $\y$-weight basis for $\sP_{as}^{+}.$

\begin{remark*}
     Note importantly that $\widetilde{E}_{(\mu|\lambda)} \in \mathscr{P}(\ell(\mu))^{+}$ and $\widetilde{E}_{(\mu|\lambda)}$ is homogeneous of degree $|\mu| + |\lambda|$. Further, we have $\widetilde{E}_{(\mu|\emptyset)} = \widetilde{E}_{\mu}  $ and $\widetilde{E}_{(\emptyset |\lambda)} = \mathcal{A}_{\lambda}$. Notice that in Definition \ref{defn9} it makes sense to consider $\widetilde{E}_{(\mu|\lambda)}$ when $\mu$ is not necessarily reduced. However, it is a nontrivial consequence of Theorem \ref{second main theorem} that an analogously defined $\widetilde{E}_{(\mu*0|\lambda)}$ is a nonzero scalar multiple of $\widetilde{E}_{(\mu|\lambda)}.$ Thus there is no need to consider the case of $\mu$ non-reduced when building a basis of $\sP_{as}^{+}.$

     There is another basis of $\sP_{as}^{+}$ given by Ion and Wu in their unpublished work \cite{IW_2022} which is equipped with a natural ordering with respect to which the limit Cherednik operators are triangular. It follows then that after we show in Corollary \ref{stable macdonald are weight vectors} that the $\widetilde{E}_{(\mu|\lambda)}$ are $\y$-weight vectors that each $\widetilde{E}_{(\mu|\lambda)}$ has a triangular expansion in Ion and Wu's basis.
\end{remark*}

\begin{remark*}
    The stable-limit non-symmetric Macdonald functions $\widetilde{E}_{(\mu|\lambda)}$ as defined in this paper are distinct from the stable-limits of non-symmetric Macdonald polynomials occurring in \cite{haglund2007combinatorial}. In their paper Haglund, Haiman, and Loehr investigate stable-limits of the form $(E_{0^m* \mu})_{m \geq 0}$ where $\mu$ is a composition. Their analysis does not require the convergence definition of Ion and Wu as the sequences $(E_{0^m* \mu})_{m \geq 0}$ have stable limits in the traditional sense. Further, the limits of the $(E_{0^m* \mu})_{m \geq 0}$ sequences are symmetric functions whereas, as we will see soon, the $\widetilde{E}_{(\mu|\lambda)}$ are not fully symmetric in general. 
\end{remark*}

The following simple lemma will be used to show that since the $\widetilde{E}_{\mu*\lambda}$ are $\y$-weight vectors the stable-limit non-symmetric Macdonald functions $\widetilde{E}_{(\mu|\lambda)}$ are $\y$-weight vectors as well. We describe their weights in Corollary \ref{stable macdonald are weight vectors}.

\begin{lem} \label{weight vectors and lowering}
Suppose $f \in \mathcal{P}(k)^{+}$ is a $\y$-weight vector with weight $(\alpha_1,\ldots,\alpha_k,0,0,\ldots)$. Then $\partial_{-}^{(k)} f \in \mathcal{P}(k-1)^{+}$ is a $\y$-weight vector with weight $(\alpha_1,\ldots,\alpha_{k-1},0,0,\ldots)$. 
\end{lem}

\begin{proof}
We know that from \cite{Ion_2022} for $g \in \mathcal{P}(k)^{+}$ and $1\leq i \leq k-1$, $\y_i \partial_{-}^{(k)} g = \partial_{-}^{(k)}\y_i g$ so $\y_i\partial_{-}^{(k)} f = \partial_{-}^{(k)}\y_i f = \alpha_i \partial_{-}^{(k)} f.$
From \cite{Ion_2022} we have that if $i\geq k$ then $\y_i$ annihilates $\mathscr{P}(k-1)$.
Since $\partial_{-}^{(k)}f \in \mathcal{P}(k-1)^{+}$ for all $i \geq k$, $\y_i \partial_{-}^{(k)}f = 0$.

\end{proof}

\begin{example}
Here we give a few basic examples of stable-limit non-symmetric Macdonald functions expanded in the Hall-Littlewood basis $\mathcal{P}_{\lambda}$ and their corresponding weights.
\begin{align*}
    &\bullet  \widetilde{E}_{(\emptyset|2)} &&= \mathcal{P}_{2}[x_1+\ldots] +\frac{q^{-1}}{1-q^{-1}t}\mathcal{P}_{1,1}[x_1+\ldots]; && \text{weight } \widetilde{\alpha}_{(\emptyset|2)} = (0,0,\ldots) \\
    &\bullet \widetilde{E}_{(2|\emptyset)} &&= x_1^2 + \frac{q^{-1}}{1-q^{-1}t}x_1 \mathcal{P}_1[x_2+\ldots]; && \text{weight } \widetilde{\alpha}_{(2|\emptyset)} =(q^2t,0,\ldots) \\
    &\bullet \widetilde{E}_{(1,1,1|\emptyset)} &&= x_1x_2x_3; && \text{weight } \widetilde{\alpha}_{(1,1,1|\emptyset)} =(qt^3,qt^2,qt,0,\ldots) \\
    &\bullet \widetilde{E}_{(1,1|1)} &&= x_1x_2\mathcal{P}_1[x_3+\ldots]; && \text{weight }  \widetilde{\alpha}_{(1,1|1)}= (qt^3,qt^2,0,\ldots) \\
    &\bullet \widetilde{E}_{(1|1,1)} &&= x_1\mathcal{P}_{1,1}[x_2+\cdots]; && \text{weight } \widetilde{\alpha}_{(1|1,1)}= (qt^3,0,\ldots)
\end{align*}
\end{example}

As an immediate result of Lemma \ref{weight vectors and lowering} we have the following:
\begin{cor}\label{stable macdonald are weight vectors}
    For $\mu \in \Compred$ and $\lambda \in \Par$,  $\widetilde{E}_{(\mu|\lambda)} \in \sP_{as}^{+}$ is a $\y$-weight vector with weight $\widetilde{\alpha}_{(\mu|\lambda)}$ given explicitly by

\[ \widetilde{\alpha}_{(\mu|\lambda)}(i) = 
    \begin{cases}
    \widetilde{\alpha}_{\mu*\lambda}(i) = q^{\mu_i}t^{\ell(\mu)+\ell(\lambda)+1-\beta_{\mu*\lambda}(i)} & i\leq \ell(\mu) , \mu_i \neq 0 \\
    0 & \text{otherwise}.
     \end{cases}
\]

\end{cor}
\begin{proof}
   By Definition \ref{defn9} we have that $$\widetilde{E}_{(\mu|\lambda)} :=  \partial_{-}^{(\ell(\mu)+1)}\cdots \partial_{-}^{(\ell(\mu)+\ell(
\lambda))} \widetilde{E}_{\mu*\lambda}. $$
From Theorem \ref{weight vector for all} we know that $\widetilde{E}_{\mu*\lambda}$ is a $\y$-weight vector with weight $\widetilde{\alpha}_{\mu*\lambda}.$ Recall that from Lemma \ref{convergence of eigenvalues} that $\widetilde{\alpha}_{\mu*\lambda}(i) = q^{(\mu*\lambda)_i}t^{\ell(\mu*\lambda)+1-\beta_{\mu*\lambda}(i)}$ for $i \leq \ell(\mu*\lambda)$ and equals $0$ for $i > \ell(\mu*\lambda).$ Using Lemma \ref{weight vectors and lowering} inductively now shows that $\widetilde{E}_{(\mu|\lambda)}$ is a $\y$-weight vector with weight $\widetilde{\alpha}_{(\mu|\lambda)}$ given by the expression given in the statement of this corollary. 
\end{proof}

By using the HHL-type formula we proved for the functions $\widetilde{E}_{\mu}$ in Corollary \ref{convergence of macdonald}, we readily find a similar formula for the full set of stable-limit non-symmetric Macdonald functions.

\begin{cor}\label{formula for stable macdonald}
    For a reduced composition $\mu$ and partition $\lambda$ we have that
    
    \begin{multline*}
        \widetilde{E}_{(\mu|\lambda)} = \sum_{\substack{\nu ~ \text{partition}  \\ |\nu| \leq |\mu|+|\lambda|}} \sum_{\substack{\sigma:\mu*\lambda * 0^{\ell(\nu)} \rightarrow [\ell(\mu)+\ell(\lambda)+\ell(\nu)]\\ \text{non-attacking} \\ \forall i = 1,...,\ell(\nu) \\ \nu_i = |\sigma^{-1}(\ell(\mu)+\ell(\lambda)+i)|}}\widetilde{\Gamma}(\widehat{\sigma}) x_1 ^{|\sigma^{-1}(1)|}\cdots x_{\ell(\mu)} ^{|\sigma^{-1}(\ell(\mu))|} ~\times \\
        \mathscr{B}_{|\sigma^{-1}(\ell(\mu)+1)|}\cdots \mathscr{B}_{|\sigma^{-1}(\ell(\mu)+\ell(\lambda))|}( m_{\nu})[\mathfrak{X}_{\ell(\mu)+ \ell(\lambda)}]
    \end{multline*} 

where 

$$\widetilde{\Gamma}(\widehat{\sigma}) := q^{-\maj(\widehat{\sigma})}t^{\coinv(\widehat{\sigma})} \prod_{\substack{ u \in dg'(\mu*\lambda * 0^{\ell(\nu)}) \\ \widehat{\sigma}(u) \neq \widehat{\sigma}(d(u)) \\ u ~ \text{not in row } 1 }} \left( \frac{1-t}{1-q^{-(\mathleg(u)+1)}t^{(a(u)+1)}} \right) \prod_{\substack{ u \in dg'(\mu * \lambda * 0^{\ell(\nu)}) \\ \widehat{\sigma}(u) \neq \widehat{\sigma}(d(u)) \\ u ~ \text{in row } 1 }} \left( 1-t \right) .$$

\end{cor}

Unfortunately, this formula is not nearly as elegant or useful as the HHL formula (\ref{HHL}). The main obstruction comes from not having a full understanding of the action of the Jing operators $\mathscr{B}_a$ on the monomial symmetric functions. If one were to find an explicit expansion of elements like $\mathscr{B}_{a_1}\cdots \mathscr{B}_{a_r}(m_{\lambda})$ into another suitable basis of $\Lambda$ (possibly the $\mathcal{P}_{\nu}$ basis) one would be able to give a much more elegant description of these functions. Likely there is a nice way to do this that has eluded this author.

\subsection{$\mathcal{A}_{\lambda}$ Basis for $\Lambda$ and Symmetrization via the Trivial Hecke Idempotent}\label{symmetrization subsection}

Lemma \ref{lowering is a projection} shows that the following operator is well defined on $\sP_{as}^{+}$ i.e. independent of $k$.
\begin{defn} \label{defn8}

For $f\in \mathscr{P}(k)^{+} \subset \mathscr{P}_{as}^{+}$  define 
\begin{equation}\label{stable limit symmetrization}
    \widetilde{\sigma}(f) := \partial_{-}^{(1)}\cdots \partial_{-}^{(k)} f .
\end{equation} Then $\widetilde{\sigma}$ defines an operator $\mathscr{P}_{as}^{+} \rightarrow \Lambda$ which we call the \textbf{\textit{stable-limit symmetrization operator}}. 
\end{defn}

\begin{remark*}
    Note that $\widetilde{\sigma}( \widetilde{E}_{ \lambda}) = \mathcal{A}_{\lambda}$ and $\widetilde{\sigma}(\widetilde{E}_{(\mu|\lambda)}) = \widetilde{\sigma}(\widetilde{E}_{\mu*\lambda}).$ 
\end{remark*}

\begin{defn}\label{idempotents}
For all $0 \leq k < n$ define the operator $\epsilon^{(n)}_k: \sP_n^{+} \rightarrow \sP_n^{+}$ as 
\begin{equation}\label{idempotent eq version 1}
\epsilon^{(n)}_k(f) : = \frac{1}{[n-k]_{t}!} \sum_{\sigma \in \mathfrak{S}_{(1^k,n-k)}} t^{{n-k \choose 2} - \ell(\sigma)} T_{\sigma}(f).
\end{equation}
Here $\mathfrak{S}_{(1^k,n-k)}$ is the Young subgroup of $\mathfrak{S}_{n}$ corresponding to the composition $(1^k,n-k)$, $T_{\sigma} = T_{s_{i_1}}\cdots T_{s_{i_r}}$ whenever $\sigma = s_{i_1}\cdots s_{i_r}$ is a reduced word representing $\sigma$, and $[m]_{t}! : = \prod_{i=1}^{m}(\frac{1-t^i}{1-t})$ is the $t$-factorial. We will simply write $\epsilon^{(n)}$ for $\epsilon^{(n)}_0$.

For $n \geq 1$ define the rational function 

\begin{equation}\label{HL kernel}
    \Omega_n(x) = \Omega_n(x_1,\dots,x_n;t): = \prod_{1\leq i<j\leq n}(\frac{x_i - tx_j}{x_i - x_j}).
\end{equation}

\end{defn}

We will need the following technical result relating the action of $\epsilon^{(n)}$ on polynomials to a Weyl character type sum involving $\Omega_n$.

\begin{prop}\label{idempotent from HL kernel}
    For $f(x) \in \sP_n^{+}$ 
    \begin{equation}
        \epsilon^{(n)}(f(x)) = \frac{1}{[n]_{t}!} \sum_{\sigma \in \mathfrak{S}_n} \sigma( f(x)\Omega_n(x) ).
    \end{equation}
\end{prop}
\begin{proof}
   See Remark 4.17 in \cite{ram2006alcove}. After translating the finite Hecke algebra quadratic relations in \cite{ram2006alcove} to match those occurring in this paper the formula matches.
\end{proof}

From the formula above in Proposition \ref{idempotent from HL kernel} we can show that the sequence of trivial idempotents $(\epsilon^{(n)})_{n\geq 1}$ converges in the sense of \cite{Ion_2022}. 

\begin{prop}\label{idempotents converge}
    The sequence of operators $(\epsilon^{(n)})_{n\geq 1}$ converges to an idempotent operator $\epsilon: \sP_{as}^{+} \rightarrow \Lambda$ such that for all $i \geq 1$, $\epsilon T_i = \epsilon$.
\end{prop}
\begin{proof}
    From \cite{Macdonald} in Chapter 3 and Proposition \ref{idempotent from HL kernel} we see that for all partitions $\lambda$ with $\ell(\lambda) = k$ and $n \geq k$ that 
    \begin{equation}
        \epsilon^{(n)}(x^{\lambda}) = \frac{[n-k]_{t}!}{[n]_{t}!} v_{\lambda}(t) P_{\lambda}[x_1+\ldots + x_n;t]
    \end{equation}
where $P_{\lambda}[X;t]$ is the Hall-Littlewood symmetric function defined by Macdonald (not to be confused with $\mathcal{P}_{\lambda}[X]$ seen previously in this paper) and 
$v_{\lambda}(t) := \prod_{i \geq 1}([m_i(\lambda)]_{t}!)$ where $m_i(\lambda)$ is the number of $i$
's in $\lambda = 1^{m_1(\lambda)}2^{m_2(\lambda)}\cdots$. Now we note that with respect to the $t$-adic topology,
$$\lim_{n \rightarrow \infty} \frac{[n-k]_{t}!}{[n]_{t}!} = (1-t)^k$$ so that 
$$ \lim_{n} \epsilon^{(n)}(x^{\lambda}) = v_{\lambda}(t)(1-t)^{\ell(\lambda)}P_{\lambda}[X;t] $$
and hence $(\epsilon^{(n)}(x^{\lambda}))_{n \geq 1}$ converges.
Note that following Macdonald's definitions, 
$$v_{\lambda}(t)(1-t)^{\ell(\lambda)}P_{\lambda}[X;t] = Q_{\lambda}[X;t].$$
Since $\epsilon^{(n)}T_i = \epsilon^{(n)}$ for $1\leq i \leq n-1$ it follows that for all compositions $\mu$, the sequence $(\epsilon^{(n)}(x^{\mu}))_{n \geq 1}$ is convergent. Clearly from definition we have that for all symmetric functions $F \in \Lambda$ and $f(x) \in \sP_n^{+}$
$$\epsilon^{(n)}(F[x_1+ \ldots + x_n]f(x)) = F[x_1+ \ldots + x_n] \epsilon^{(n)}(f(x)).$$ It follows now from a straightforward convergence argument using Remark \ref{useful defn of convergence} that for all $g \in \sP_{as}^{+}$ the sequence $(\epsilon^{(n)}(\pi_n(g)))_{n\geq 1}$ converges. The resulting operator $\epsilon:= \lim_{n} \epsilon^{(n)}\circ \pi_n$ is evidently idempotent as its codomain is $\Lambda$ and certainly $\epsilon$ acts as the identity on symmetric functions. Further, for all $i \in \mathbb{N}$ we have  
$$ \epsilon T_i = \lim_{n} \epsilon^{(n)}\circ \pi_n T_i $$
and since $\pi_n$ commutes with $T_i$ for $n > i + 1$ we see that 
$$\lim_{n} \epsilon^{(n)}\circ \pi_n T_i = \lim_{n} \epsilon^{(n)} T_i \circ \pi_n = \lim_{n} \epsilon^{(n)} \circ \pi_n  = \epsilon.$$
\end{proof}

\begin{cor}\label{partial idempotents converge}
    For all $k \geq 0$ the sequence $(\epsilon^{(n)}_k)_{n>k}$ converges to an idempotent operator $\epsilon_k: \sP_{as}^{+} \rightarrow \sP(k)^{+}$ such that for all $ i \geq k+1$, $\epsilon_k T_i = \epsilon_k$.
\end{cor}

\begin{proof}
    This follows immediately from Proposition \ref{idempotents converge} after shifting indices and noting that the operators $\epsilon^{(n)}_k$ commute with multiplication by $x_1,\dots, x_k$.
\end{proof}

Now we will extend our definition of the stable-limit symmetrization operator $\widetilde{\sigma}$ to partial symmetrization operators in the natural way.

\begin{defn}\label{defn 12}
    For $k \geq 0$ let $\widetilde{\sigma}_k: \sP_{as}^{+} \rightarrow \sP(k)^{+}$ be defined on $g \in \sP(n)^{+}$ for $n \geq k$ by 
    \begin{equation}
        \widetilde{\sigma}_k(g):= \partial_{-}^{(k+1)}\cdots \partial_{-}^{(n)}(g).
    \end{equation}
\end{defn}

\begin{remark*}
    The operators $\widetilde{\sigma}_k$ are well defined by Lemma \ref{lowering is a projection}. In particular, if $g \in \sP(\ell)^{+}$ for $0 \leq \ell \leq k$ then $\sP(\ell)^{+} \subset \sP(k)^{+}$ and there is no ambiguity in defining $\widetilde{\sigma}_k(g) = \partial_{-}^{(k+1)}\cdots \partial_{-}^{(n)}(g)$ as above. Note that $\widetilde{\sigma}_0 = \widetilde{\sigma}$. Further, for all $\mu \in \Compred$ and $\lambda \in \Par$ we see that in this new terminology 
    $$\widetilde{E}_{(\mu|\lambda)} = \widetilde{\sigma}_{\ell(\mu)}(\widetilde{E}_{\mu*\lambda}). $$ Further, if $k \leq \ell$ then $\widetilde{\sigma}_k\widetilde{\sigma}_{\ell} = \widetilde{\sigma}_k.$
\end{remark*}

We will now show that as operators on $\sP_{as}^{+}$, $\epsilon_{\ell} = \widetilde{\sigma}_{\ell}$ for all ${\ell} \geq 0$.

\begin{prop}\label{idempotent = HL}
    For all $\ell \geq 0$, $\epsilon_{\ell} = \widetilde{\sigma}_{\ell} .$
\end{prop}

\begin{proof}
    By shifting indices it suffices to just prove that $\epsilon = \widetilde{\sigma}$, i.e., the $\ell = 0$ case. Further, since both maps are $T_i$-equivariant $\Lambda$-module maps (see Corollary \ref{lowering is a sym module map}) it suffices to show that for all partitions $\lambda$, $\epsilon(x^{\lambda}) = \widetilde{\sigma}(x^{\lambda}).$ From the proof of Proposition \ref{idempotents converge} we saw that $\epsilon(x^{\lambda}) = Q_{\lambda}[X;t]$ whereas it follows from the definition of the Jing vertex operators that $\widetilde{\sigma}(x^{\lambda}) = \mathcal{P}_{\lambda}[X]$. Therefore, it suffices to argue that $Q_{\lambda}[X;t] = \mathcal{P}_{\lambda}[X].$ To this end we will prove that 

    \begin{equation}\label{symmetrization-plthysm}
        \mathcal{P}_{\lambda}[X] = \langle z_1^{\lambda_1}\cdots z_r^{\lambda_r} \rangle Exp[(1-t)(z_1+\ldots+z_r)X]Exp[(t-1)\sum_{1\leq i < j \leq r} \frac{z_j}{z_i}]
    \end{equation}
 which by 2.15 in Macdonald Chapter 3 \cite{Macdonald} is an alternative definition for $Q_{\lambda}[X;t].$

Suppose $\lambda = (\lambda_1,\ldots, \lambda_r)$ is a partition.    
Note first that by definition $ \mathcal{P}_{\lambda}[X] = \mathscr{B}_{\lambda_1}\cdots \mathscr{B}_{\lambda_r}(1).$ We will now induct on the number of operators $\mathscr{B}$ acting on $1$ in the expression $\mathscr{B}_{\lambda_1}\cdots \mathscr{B}_{\lambda_r}(1).$ As a base case  
$$\mathscr{B}_{\lambda_r}(1) = \langle z_r ^{\lambda_r} \rangle 1[X-z_{r}^{-1}]Exp[(1-t)z_rX] = \langle z_r ^{\lambda_r} \rangle Exp[(1-t)z_rX].$$

We claim that for all $1\leq k \leq r$

\begin{equation} \label{symmetrization-plethysm-induction-equation}
    \mathscr{B}_{\lambda_k}\cdots \mathscr{B}_{\lambda_r}(1) = \langle z_{k}^{\lambda_{k}} \cdots z_r^{\lambda_r} \rangle Exp[(1-t)(z_k+\ldots +z_r)X]Exp[(t-1)\sum_{k\leq i<j\leq r}\frac{z_j}{z_i}].
\end{equation}

Suppose the above is true for some $1< k \leq r$. Then 

\begin{align*}
    &\mathscr{B}_{\lambda_{k-1}}\mathscr{B}_{\lambda_k}\cdots \mathscr{B}_{\lambda_r}(1)  \\
    &= \mathscr{B}_{\lambda_{k-1}} \left( \langle z_{k}^{\lambda_{k}} \cdots z_r^{\lambda_r} \rangle Exp[(1-t)(z_k+\ldots +z_r)X]Exp[(t-1)\sum_{k\leq i<j\leq r}\frac{z_j}{z_i}] \right) \\
   &=  \langle z_{k-1}^{\lambda_{k-1}} \rangle  \langle z_{k}^{\lambda_{k}} \cdots z_r^{\lambda_r} \rangle Exp[(1-t)(z_k+\ldots +z_r)(X-z_{k-1}^{-1})]Exp[(t-1)\sum_{k\leq i<j\leq r}\frac{z_j}{z_i}] Exp[(1-t)z_{k-1}X]. \\
\end{align*}

Now we use the additive property of the plethystic exponential namely, $Exp[A+B] = Exp[A]Exp[B]$, to rearrange terms and get
    $$ \langle z_{k-1}^{\lambda_{k-1}} \cdots z_r^{\lambda_r} \rangle Exp[(1-t)(z_k+\ldots +z_r)X] Exp[(1-t)z_{k-1}X]Exp[(t-1)\sum_{k\leq i<j\leq r}\frac{z_j}{z_i}]Exp[(t-1)(\frac{z_k}{z_{k-1}}+\ldots +\frac{z_r}{z_{k-1}})] $$
    which simplifies to 
$$ \langle z_{k-1}^{\lambda_{k-1}} \cdots z_r^{\lambda_r} \rangle Exp[(1-t)(z_{k-1}+z_k+\ldots +z_r)X] Exp[(t-1)\sum_{k-1\leq i<j\leq r}\frac{z_j}{z_i}] $$ showing that the formula (\ref{symmetrization-plethysm-induction-equation}) holds for all $1\leq k\leq r$. Taking $k=1$ shows equation (\ref{symmetrization-plthysm}) holds.
\end{proof}

As an immediate consequence of Proposition \ref{idempotent = HL} we find the following enlightening description for the $\widetilde{E}_{(\mu|\lambda)}$ functions.

\begin{cor}\label{alternative description of E's}\label{A lambdas are basis}
    For all $(\mu|\lambda)$ with $\mu$ a reduced composition and $\lambda$ a partition,
    \begin{equation}
        \widetilde{E}_{(\mu|\lambda)} = \lim_{n} \epsilon_{\ell(\mu)}^{(n)}(E_{\mu*\lambda*0^{n-(\ell(\mu)+\ell(\lambda))}}).
    \end{equation}
    In particular, for partitions $\lambda$, $\mathcal{A}_{\lambda}[X] = (1-t)^{\ell(\lambda)} v_{\lambda}(t)P_{\lambda}[X;q^{-1},t]$ where $P_{\lambda}[X;q^{-1},t]$ is the symmetric Macdonald function. As a consequence the set $\{ \mathcal{A}_{\lambda} : \lambda \in \Par \}$ is a basis of $\Lambda$.
\end{cor}

\begin{remark*}
    The $P_{\lambda}[X;q,t]$ are the symmetric Macdonald functions as defined by Macdonald in \cite{Macdonald} and seen in Cherednik's work \cite{C_2001} not to be confused with the modified symmetric Macdonald functions $\widetilde{H}_{\mu}$ seen in many places but in particular in the work of Haiman \cite{HVanish}. Further, Corollary \ref{alternative description of E's} gives an interpretation of the $\widetilde{E}_{(\mu|\lambda)}$ as limits of partially symmetrized non-symmetric Macdonald polynomials. Goodberry in \cite{Goodberry} and Lapointe in \cite{lapointe2022msymmetric} have investigated similar families of partially symmetric Macdonald polynomials. Up to a change of variables and limiting these different notions are likely directly related. 
\end{remark*}

In order to prove the first main theorem in this paper, Theorem \ref{main theorem}, we will require the following straightforward lemma. 

\begin{lem} \label{permuting and A lambdas}
For any composition $\mu$ there is some nonzero scalar $\gamma_{\mu} \in \mathbb{Q}(q,t)$ such that 
$$ \widetilde{\sigma}( \widetilde{E}_{\mu} ) = \gamma_{\mu} \mathcal{A}_{\sort(\mu)} $$
where $\gamma_{\mu} = 1$ when $\mu$ is a partition.
\end{lem}

\begin{proof}

We know that for all partitions $\lambda$,  $\widetilde{\sigma}( \widetilde{E}_{\lambda} ) = \mathcal{A}_{\lambda}$ so this lemma holds trivially for partitions. Now we proceed by induction on Bruhat order similarly to the argument in the proof of Theorem \ref{weight vector for all}. To show the lemma holds it suffices to show that if $\mu$ is a composition and $k$ such that $s_k(\mu) > \mu$ in Bruhat order and $\widetilde{\sigma}( \widetilde{E}_{\mu} ) = \gamma_{\mu} \mathcal{A}_{\sort(\mu)}$ for $\gamma_{\mu} \neq 0$ then $\widetilde{\sigma}( \widetilde{E}_{s_k(\mu)} ) = \gamma_{s_k(\mu)} \mathcal{A}_{\sort(\mu)}$ for $\gamma_{s_k(\mu)} \neq 0$. To this end fix such $\mu$ and $k$. Then by Corollary \ref{action of T on E mu}
$$\widetilde{E}_{s_k(\mu)} = \left( T_k + \frac{(1-t)\widetilde{\alpha}_{\mu}(k+1)}{\widetilde{\alpha}_{\mu}(k)- \widetilde{\alpha}_{\mu}(k+1)}\right) \widetilde{E}_{\mu}.$$
From Proposition \ref{idempotent = HL} $\widetilde{\sigma} = \lim_{m} \epsilon^{(m)}$ so that $\widetilde{\sigma}T_k = \widetilde{\sigma}$. Therefore, 

\begin{align*}
    \widetilde{\sigma}(\widetilde{E}_{s_k(\mu)})&= \widetilde{\sigma}\left( \left( T_k + \frac{(1-t)\widetilde{\alpha}_{\mu}(k+1)}{\widetilde{\alpha}_{\mu}(k)- \widetilde{\alpha}_{\mu}(k+1)}\right) \widetilde{E}_{\mu} \right) \\
    &= \left(1+ \frac{(1-t)\widetilde{\alpha}_{\mu}(k+1)}{\widetilde{\alpha}_{\mu}(k)- \widetilde{\alpha}_{\mu}(k+1)} \right) \widetilde{\sigma}(\widetilde{E}_{\mu}) \\
    &= \left( \frac{\widetilde{\alpha}_{\mu}(k) - t \widetilde{\alpha}_{\mu}(k+1)}{\widetilde{\alpha}_{\mu}(k) - \widetilde{\alpha}_{\mu}(k+1)} \right) \gamma_{\mu} \mathcal{A}_{\sort(\mu)}. \\
\end{align*}

By Lemma \ref{convergence of eigenvalues} we see that since $s_k(\mu) > \mu$ it follows that $\widetilde{\alpha}_{\mu}(k) \neq t \widetilde{\alpha}_{\mu}(k+1).$
Hence, $\gamma_{s_k(\mu)} := \left( \frac{\widetilde{\alpha}_{\mu}(k) - t \widetilde{\alpha}_{\mu}(k+1)}{\widetilde{\alpha}_{\mu}(k) - \widetilde{\alpha}_{\mu}(k+1)} \right) \gamma_{\mu} \neq 0$ so the result follows.
\end{proof}

\begin{remark*}
    Note that using the recursive formula $\gamma_{s_k(\mu)} = \left( \frac{\widetilde{\alpha}_{\mu}(k) - t \widetilde{\alpha}_{\mu}(k+1)}{\widetilde{\alpha}_{\mu}(k) - \widetilde{\alpha}_{\mu}(k+1)} \right) \gamma_{\mu}$ in the proof of Lemma \ref{permuting and A lambdas}, the formula for the eigenvalues $\widetilde{\alpha}_{\mu}(k)$ in Lemma \ref{convergence of eigenvalues}, and the base condition $\gamma_{\mu} = 1$ for $\mu$ a partition, it is possible to give an explicit expression for $\gamma_{\mu}$ for any composition $\mu.$ However, all we need for the purposes of this paper is that $\gamma_{\mu} \neq 0$ so we will not find such an explicit expression for $\gamma_{\mu}.$
\end{remark*}

\subsubsection{First Main Theorem and a Full $\y$-Weight Basis of $\sP_{as}^{+}$}

Finally, we prove that the stable-limit non-symmetric Macdonald functions are a basis for $\mathscr{P}_{as}^{+}$. To do this we will use the stable-limit symmetrization operator to help distinguish between stable-limit non-symmetric Macdonald functions with the same $\y$-weight. 

\begin{thm}(First Main Theorem) \label{main theorem}
The $\widetilde{E}_{(\mu|\lambda)}$ are a $\y$-weight basis for $\mathscr{P}_{as}^{+}$.
\end{thm}

\begin{proof}
As there are sufficiently many $\widetilde{E}_{(\mu|\lambda)}$ in each graded component of every $\mathscr{P}(k)^{+}$ it suffices to show that these functions are linearly independent. Certainly, weight vectors in distinct weight spaces are linearly independent. Using Lemmas \ref{convergence of eigenvalues} and \ref{stable macdonald are weight vectors}, we deduce that if $\widetilde{E}_{(\mu^{(1)}|\lambda^{(1)})}$ and $\widetilde{E}_{(\mu^{(2)}|\lambda^{(2)})}$ have the same weight then necessarily $\mu^{(1)} = \mu^{(2)}$. Hence, we can restrict to the case where we
have a dependence relation
$$ c_1\widetilde{E}_{(\mu|\lambda^{(1)})} +\ldots + c_N \widetilde{E}_{(\mu|\lambda^{(N)})} = 0 $$
for $\lambda^{(1)},\ldots,\lambda^{(N)}$ distinct partitions.
By applying the stable-limit symmetrization operator we see that 
$$ \widetilde{\sigma}(c_1\widetilde{E}_{(\mu|\lambda^{(1)})} +\ldots + c_N \widetilde{E}_{(\mu|\lambda^{(N)})} ) = \widetilde{\sigma}(c_1\widetilde{E}_{\mu*\lambda^{(1)}} +\ldots+c_N \widetilde{E}_{\mu*\lambda^{(N)}}) = 0 .$$
Now by Lemma \ref{permuting and A lambdas}, $\widetilde{\sigma}(\widetilde{E}_{\mu*\lambda^{(i)}}) = \gamma_{\mu*\lambda^{(i)}}  \mathcal{A}_{\sort(\mu*\lambda^{(i)})} $ with nonzero scalars  $\gamma_{\mu*\lambda^{(i)}}$ yielding 
$$ 0 = c'_1\mathcal{A}_{\sort(\mu*\lambda^{(1)})} +\ldots+ c'_n \mathcal{A}_{\sort(\mu*\lambda^{(N)})}.$$ The partitions $\lambda^{(i)}$ are distinct so we know that the partitions $\sort(\mu*\lambda^{(i)})$ are distinct as well. By Corollary \ref{A lambdas are basis} the symmetric functions $\mathcal{A}_{\sort(\mu*\lambda^{(i)})}$ are linearly independent. Thus $c'_i = 0$ implying $c_i = 0$ for all $1 
\leq i \leq N$ as desired.
\end{proof}

\section{Some Recurrence Relations for the $\widetilde{E}_{(\mu|\lambda)}$}\label{recurrence relations section}

In this section we will discuss a few recurrence relations for the stable-limit non-symmetric Macdonald functions. We start by looking at the action of the Demazure-Lusztig operators $T_i$ and the lowering operators $\partial_{-}$.

\begin{prop}
    For a reduced composition $\mu = (\mu_1,\dots,\mu_r)$ and partition $\lambda =(\lambda_1,\dots,\lambda_k)$ if $\mu_r \geq \lambda_1$ and $\mu_{r-1} \neq 0$ then 
    $$ \partial_{-}^{(r)} \left( \widetilde{E}_{(\mu_1,\dots,\mu_r|\lambda_1,\dots,\lambda_k)} \right) = \widetilde{E}_{(\mu_1,\dots,\mu_{r-1}|\mu_r,\lambda_1,\dots,\lambda_k)}.$$
\end{prop}
\begin{proof}
    This follows immediately from the definitions of $\widetilde{E}_{(\mu|\lambda)}$ and $\partial_{-}^{(r)}.$
\end{proof}

\begin{prop}
Take $\mu \in \Compred$ and $\lambda \in \Par$ and suppose $1\leq i \leq \ell(\mu) -1$ such that $s_i(\mu) > \mu$ and $s_i(\mu) \in \Compred.$  Then 
$$\widetilde{E}_{(s_i(\mu)|\lambda)} = 
  \left( T_i + \frac{(1-t)\widetilde{\alpha}_{\mu*\lambda}(i+1)}{\widetilde{\alpha}_{\mu*\lambda}(i)- \widetilde{\alpha}_{\mu*\lambda}(i+1)} \right)  \widetilde{E}_{(\mu|\lambda)}.$$
\end{prop}

\begin{proof}
Since $s_i(\mu) > \mu$ we know that $s_i(\mu*\lambda) > \mu*\lambda$ so by Corollary \ref{action of T on E mu}
$$\widetilde{E}_{s_i(\mu*\lambda)} = 
  \left( T_i + \frac{(1-t)\widetilde{\alpha}_{\mu*\lambda}(i+1)}{\widetilde{\alpha}_{\mu*\lambda}(i)- \widetilde{\alpha}_{\mu*\lambda}(i+1)} \right)  \widetilde{E}_{\mu*\lambda}. $$
Now we know $T_i$ commutes with the operators $\partial_{-}^{(\ell(\mu)+1)},\dots , \partial_{-}^{(\ell(\mu)+\ell(
\lambda))}$ and thus we see that 
\begin{align*}
    \widetilde{E}_{(s_i(\mu)|\lambda)}&= \partial_{-}^{(\ell(\mu)+1)}\cdots \partial_{-}^{(\ell(\mu)+\ell(
\lambda))} (\widetilde{E}_{s_i(\mu*\lambda)}) \\
&= \partial_{-}^{(\ell(\mu)+1)}\cdots \partial_{-}^{(\ell(\mu)+\ell(
\lambda))} \left( \left( T_i + \frac{(1-t)\widetilde{\alpha}_{\mu*\lambda}(i+1)}{\widetilde{\alpha}_{\mu*\lambda}(i)- \widetilde{\alpha}_{\mu*\lambda}(i+1)} \right)  \widetilde{E}_{\mu*\lambda}\right) \\
&= \left( T_i + \frac{(1-t)\widetilde{\alpha}_{\mu*\lambda}(i+1)}{\widetilde{\alpha}_{\mu*\lambda}(i)- \widetilde{\alpha}_{\mu*\lambda}(i+1)} \right) \partial_{-}^{(\ell(\mu)+1)}\cdots \partial_{-}^{(\ell(\mu)+\ell(
\lambda))}(\widetilde{E}_{\mu*\lambda}) \\
&= \left( T_i + \frac{(1-t)\widetilde{\alpha}_{\mu*\lambda}(i+1)}{\widetilde{\alpha}_{\mu*\lambda}(i)- \widetilde{\alpha}_{\mu*\lambda}(i+1)} \right) \widetilde{E}_{(\mu|\lambda)}.
\end{align*}

\end{proof}

\begin{prop}
    For $\mu = (\mu_1,\dots, \mu_r) \in \Compred$ and $\lambda \in \Par$ we have that
    $$T_r\widetilde{E}_{(\mu|\lambda)} = \frac{\gamma_{\mu*\lambda}}{\gamma_{(\mu_1,\ldots,\mu_{r-1},0,\mu_r)*\lambda}} \widetilde{E}_{(\mu_1,\dots,\mu_{r-1},0,\mu_r|\lambda)}.$$
\end{prop}

\begin{proof}
First note that by Corollary \ref{stable macdonald are weight vectors}
$$\varphi_r(\widetilde{E}_{(\mu|\lambda)}) = (T_r(\y_r-\y_{r+1})+(1-t)\y_{r+1})\widetilde{E}_{(\mu|\lambda)} = (\widetilde{\alpha}_{\mu*\lambda}(r)-0)T_r\widetilde{E}_{(\mu|\lambda)} + (1-t)(0)\widetilde{E}_{(\mu|\lambda)} = \widetilde{\alpha}_{\mu*\lambda}(r)T_r\widetilde{E}_{(\mu|\lambda)}$$
and by Lemma \ref{convergence of eigenvalues} $\widetilde{\alpha}_{\mu*\lambda}(r) \neq 0$ since $\mu_r \neq 0$. Therefore, $\varphi_r(\widetilde{E}_{(\mu|\lambda)})$ is nonzero and therefore must be a $\y$-weight vector with weight $(\widetilde{\alpha}_{\mu*\lambda}(1),\dots,\widetilde{\alpha}_{\mu*\lambda}(r-1),0,\widetilde{\alpha}_{\mu*\lambda}(r), 0,\dots)$. By using the explicit formula for the eigenvalues $\widetilde{\alpha}_{\mu*\lambda}(i)$ from Lemma \ref{convergence of eigenvalues} we see that for $1 \leq i \leq r$, $\widetilde{\alpha}_{\mu*\lambda}(i) = 0$ exactly when $\mu_i = 0$ and further, for all $1 \leq i \leq r$ with $\mu_i \neq 0$, $\widetilde{\alpha}_{\mu*\lambda}(i) = q^{\mu_i}t^{b_i}$ for some $b_i$. Hence by Theorem \ref{main theorem} and Corollary \ref{stable macdonald are weight vectors}, $\varphi_r(\widetilde{E}_{(\mu|\lambda)})$ is of the form 
$$\varphi_r(\widetilde{E}_{(\mu|\lambda)}) = \sum_{\nu} a_{\nu} \widetilde{E}_{(\mu_1,\dots,\mu_{r-1},0,\mu_r|\nu)}$$
$\nu$ ranges over a finite set of partitions $\nu$ and $a_{\nu}$ are some scalars. 
Note that we have 
$$\widetilde{\sigma}(\varphi_r(\widetilde{E}_{(\mu|\lambda)})) =  \widetilde{\sigma}(\widetilde{\alpha}_{\mu*\lambda}(r)T_r\widetilde{E}_{(\mu|\lambda)}) $$ and since $\widetilde{\sigma}T_r = \widetilde{\sigma}$ $$\widetilde{\sigma}(\varphi_r(\widetilde{E}_{(\mu|\lambda)}))= \widetilde{\alpha}_{\mu*\lambda}(r)\widetilde{\sigma}(\widetilde{E}_{(\mu|\lambda)}) = \widetilde{\alpha}_{\mu*\lambda}(r) \gamma_{\mu*\lambda} \mathcal{A}_{\sort(\mu*\lambda)}$$ 
using Lemma \ref{permuting and A lambdas}. 
Similarly, we see that 
$$\widetilde{\sigma}\left( \sum_{\nu} a_{\nu} \widetilde{E}_{(\mu_1,\dots,\mu_{r-1},0,\mu_r|\nu)}   \right) = \sum_{\nu} a_{\nu} \gamma_{(\mu_1,\dots,\mu_{r-1},0,\mu_r)*\nu} \mathcal{A}_{\sort(\mu*\nu)}$$
since $\sort((\mu_1,\ldots,\mu_{r-1},0,\mu_r)*\nu) = \sort(\mu*\nu)$ for all $\nu.$

Thus $$\mathcal{A}_{\sort(\mu*\lambda)} = \sum_{\nu} a_{\nu}' \mathcal{A}_{\sort(\mu*\nu)} $$
where $$a'_{\nu}:= \frac{a_{\nu}\gamma_{(\mu_1,\dots,\mu_{r-1},0,\mu_r)*\nu}}{\widetilde{\alpha}_{\mu*\lambda}(r)\gamma_{\mu*\lambda}}.$$ By Corollary \ref{A lambdas are basis} we know that the $\mathcal{A}_{\beta}$ are a basis for $\Lambda$ and so we see that the only possible partition $\nu$ that can contribute a nonzero term in the above expansion is $\nu = \lambda$. Further, $a'_{\lambda} = 1$ and thus $a_{\lambda} = \frac{\widetilde{\alpha}_{\mu*\lambda}(r)\gamma_{\mu*\lambda}}{\gamma_{(\mu_1,\ldots,\mu_{r-1},0,\mu_r)*\lambda}}.$

Therefore, 
$$\varphi_{r}(\widetilde{E}_{(\mu|\lambda)}) = \widetilde{\alpha}_{\mu*\lambda}(r) T_r \widetilde{E}_{(\mu|\lambda)} = \frac{\widetilde{\alpha}_{\mu*\lambda}(r)\gamma_{\mu*\lambda}}{\gamma_{(\mu_1,\ldots,\mu_{r-1},0,\mu_r)*\lambda}} \widetilde{E}_{(\mu_1,\dots,\mu_{r-1},0,\mu_r|\lambda)}$$ which yields 
$$T_r\widetilde{E}_{(\mu|\lambda)} = \frac{\gamma_{\mu*\lambda}}{\gamma_{(\mu_1,\ldots,\mu_{r-1},0,\mu_r)*\lambda}} \widetilde{E}_{(\mu_1,\dots,\mu_{r-1},0,\mu_r|\lambda)}.$$
\end{proof}

\begin{defn}\label{defn 10}
 Define $\widetilde{\omega}_m := X_1T_1^{-1}\cdots T_{m-1}^{-1}$ considered as an operator on $\sP_{m}^{+}$.
\end{defn}

\begin{remark*}
    These operators are the same as the corresponding operators of the same name defined by Ion and Wu up to inversion and some scalars. We have defined the operators as above for convenience. The operators $\omega_m$ and $\widetilde{\omega}_m$ are used by Ion and Wu \cite{Ion_2022} to give operators analogous to the $d_{+},d_{+}^{*}$ operators in $\mathbb{A}_{t,q}$.
\end{remark*}

\begin{lem}\label{pis converges}
    The sequences of operators $(\widetilde{\omega}_m)_{m\geq 1}$ and $(\omega_{m}^{-1})_{m \geq 1}$ converge to operators $\widetilde{\omega}, \omega^{*}: \sP_{as}^{+} \rightarrow \sP_{as}^{+}$ respectively with actions given by 
    \begin{itemize}
        \item $\widetilde{\omega}(x_1^{a_1}\cdots x_k^{a_k}F[X]) = x_1T_1^{-1}\cdots T_{k}^{-1}x_1^{a_1}\cdots x_k^{a_k}F[X]$ \label{tilde omega action}
        \item $\omega^{*}(x_1^{a_1}\cdots x_k^{a_k}F[X]) = x_2^{a_1}\cdots x_{k+1}^{a_k}F[X+(q-1)x_1].$
    \end{itemize}

\end{lem}

\begin{proof}
    Let $(f_m)_{m\geq 1}$ be a convergent sequence with limit $f \in \sP(k)^{+}$. We start by showing the sequence $(\widetilde{\omega}_m(f_m))_{m \geq 1}$ converges to an element of $\sP_{as}^{+}$. It follows directly by the definition of convergence that there exists some $M >k$ such that for all $i$ and $m$ with $m \geq M$ and $k+1 \leq i \leq m-1$, $T_if_m = f_m$. Therefore, for all $m \geq M$ 
    $$\widetilde{\omega}_m(f_m) = x_1T_1^{-1}\cdots T_{k}^{-1}f_m $$ 
    which clearly converges to $x_1T_1^{-1}\cdots T_{k}^{-1}f$. It follows then that the sequence of operators $(\widetilde{\omega}_m)_{m\geq 1}$ converges to an operator which we call $\widetilde{\omega}.$ By considering $f = x_1^{a_1}\cdots x_k^{a_k}F[X]$ with $F \in \Lambda$ we get the first formula 
    in the lemma statement above.

    Next we will show the sequence $(\omega^{-1}_m(\pi_m(f)))_{m \geq 1}$ converges.
    Expand $f$ as
    $$f = \sum_{i=1}^{N} c_ix^{\mu^{(i)}}F_i[X]$$
    for $c_i \in \mathbb{Q}(q,t)$, compositions $\mu^{(i)}$, and $F_i \in \Lambda$ where we may assume each composition $\mu^{(i)}$ has length $k$ so that for all $m \geq k$ 
    $$ \pi_m(f) = \sum_{i=1}^{N} c_i x^{\mu^{(i)}} F_i[x_1+\ldots + x_m].$$
    Applying $\omega_{m}^{-1}$ to $\pi_m(f)$ gives for $m \geq k$ 
    $$\omega_{m}^{-1}(\pi_m(f)) = \sum_{i=1}^{N} c_i x_2^{\mu^{(i)}_1}\cdots x_{k+1}^{\mu^{(i)}_k}F[qx_1+x_2+\ldots + x_m] $$
    so therefore we get 
    $$\lim_{m} \omega_{m}^{-1}(\pi_m(f)) = \sum_{i=1}^{N} c_i x_2^{\mu^{(i)}_1}\cdots x_{k+1}^{\mu^{(i)}_k}F[X+(q-1)x_1].$$
    Thus the sequence of operators $(\omega_{m}^{-1})_{m \geq 1}$ converges to an operator which we call $\omega^{*}$. Lastly, by applying this formula to $f = x_1^{a_1}\cdots x_k^{a_k}F[X]$ with $F \in \Lambda$ to see the second formula given in the lemma statement.
\end{proof}

In line with the above results in this section we will now give a partial generalization of the classical Knop-Sahi relation regarding the action of the $\omega$ operators on Macdonald polynomials.  

\begin{prop}
    For all compositions $\mu$ 
    $$t^{\#\{j:\mu_j \neq 0\}}\widetilde{\omega}( \widetilde{E}_{\mu}) = x_1 \omega^{*}(\widetilde{E}_{\mu}) = \widetilde{E}_{1*\mu}.$$
\end{prop}
\begin{proof}
    Suppose $\ell(\mu) = n$. Recall that for all $m \geq 1$
    $$(Y^{(n+m)}_{n+m})^{-1} = t^{n+m-1} \omega_{n+m}T_1^{-1}\cdots T_{n+m-1}^{-1}.$$ Therefore, by recalling the eigenvalue notation in Lemma \ref{convergence of eigenvalues} we have 
    $$t^{n+m-1} \omega_{n+m}T_1^{-1}\cdots T_{n+m-1}^{-1}E_{\mu*0^m} = (Y^{(n+m)}_{n+m})^{-1} E_{\mu*0^m} = \alpha_{\mu}^{(m)}(n+m)^{-1}E_{\mu*0^m}$$
    so that 
    $$t^{n+m-1}\alpha_{\mu}^{(m)}(n+m) x_1 T_1^{-1}\cdots T_{n+m-1}^{-1}E_{\mu*0^m} = x_1 \omega_{n+m}^{-1}E_{\mu*0^m}.$$
    From Lemma \ref{convergence of eigenvalues} we see that 
    $$t^{n+m-1}\alpha_{\mu}^{(m)}(n+m) = t^{\#\{j:\mu_j \neq 0\}}$$
    which gives 
    $$t^{\#\{j:\mu_j \neq 0\}} x_1 T_1^{-1}\cdots T_{n+m-1}^{-1}E_{\mu*0^m} = t^{\#\{j:\mu_j \neq 0\}} \widetilde{\omega}_{n+m}(E_{\mu*0^m}) = x_1 \omega_{n+m}^{-1}E_{\mu*0^m}. $$
    From the classical Knop-Sahi relations (see \cite{haglund2007combinatorial}) applied to $E_{\mu*0^m}$ we get 
    $x_1 \omega_{n+m}^{-1}E_{\mu*0^m} = E_{1*\mu*0^{m-1}}.$
    Applying Corollary \ref{convergence of macdonald} and Lemma \ref{pis converges} as $m \rightarrow \infty$ now gives 
    $$t^{\#\{j:\mu_j \neq 0\}}\widetilde{\omega}(\widetilde{E}_{\mu}) = x_1 \omega^{*}(\widetilde{E}_{\mu}) = \widetilde{E}_{1*\mu}.$$
    
\end{proof}

\section{Constructing $\widetilde{E}_{(\mu|\lambda)}$-Diagonal Operators from Symmetric Functions}\label{one dim wt spaces section}

The main goal of the following section of this paper is to construct an operator on $\sP_{as}^{+}$ which is diagonal in the stable-limit Macdonald function basis, commutes with the limit Cherednik operators $\y_i$, but does not annihilate $\Lambda$. This operator will be constructed from a limit of operators arising from the action of $t^mY^{(m)}_1+ \ldots +t^m Y^{(m)}_m$ on $\sP_m^{+}$. After finding the eigenvalues of this new operator we will show that the addition of this operator to the algebra generated by the limit Cherednik operators has simple spectrum on $\sP_{as}^{+}$.

We begin with the following natural definition. 

\begin{defn}\label{defn 13}
    For $F \in \Lambda$ define the operator $\Psi_{F}^{(m)}: \sP_{m}^{+} \rightarrow \sP_{m}^{+}$ by
    \begin{equation}
    \Psi_{F}^{(m)} := F[t^m Y_1^{(m)}+ \ldots + t^m Y_m^{(m)}].
    \end{equation}
    Further, for a composition $\mu$ with $\ell(\mu) = n$ and $m \geq 0$ define the scalar $\kappa_{\mu}^{(m)}(q,t)$ as 
    $$\kappa_{\mu}^{(m)}(q,t):= \sum_{i=1}^{n+m} t^{n+m} \alpha_{\mu}^{(m)}(i).$$
\end{defn}

 Recall from Lemma \ref{convergence of eigenvalues} that $\alpha_{\mu}^{(m)}(i)$ is given by $Y_{i}^{(n+m)}E_{\mu*0^m} = \alpha_{\mu}^{(m)}(i)E_{\mu*0^m}$.

\begin{lem}\label{kappas converge}
    For all compositions $\mu$ the sequence $(\kappa_{\mu}^{(m)}(q,t))_{m\geq 0}$ converges to some $\kappa_{\mu}(q,t) \in \mathbb{Q}(q,t).$ Further, $\kappa_{\mu}(q,t) = \kappa_{\mu*0^k}(q,t)$ for all $k \geq 0$ and $\kappa_{\mu}(q,t) = \kappa_{s_i(\mu)}(q,t) $ for all $1\leq i \leq \ell(\mu) -1$.
\end{lem}

\begin{proof}
    Using Lemma \ref{convergence of eigenvalues} we get the following:

\begin{align*}
    \kappa_{\mu}^{(m)}(q,t) &= \sum_{i=1}^{n+m}t^{n+m}\alpha_{\mu}^{(m)}(i) \\
    &= \sum_{i=1}^{n}t^{n+m}\alpha_{\mu}^{(m)}(i) + \sum_{i=n+1}^{n+m} t^{\#\{j:\mu_j \neq 0\}}t^{m+1-(i-n)} \\
    &= \sum_{i=1}^{n}t^n \alpha_{\mu}^{(0)}(i) t^{m \mathbbm{1}(\mu_i=0)} + t^{\#\{j:\mu_j \neq 0\}}\sum_{i=1}^{m}t^{m-i+1} \\
    &= \sum_{\mu_i \neq 0}t^n \alpha_{\mu}^{(0)}(i) + t^m \sum_{\mu_i =0} t^n \alpha_{\mu}^{(0)}(i) + t^{\#\{j:\mu_j \neq 0\}}\sum_{i=1}^{m}t^{i}.
\end{align*}

Therefore, 
\begin{equation}
\kappa_{\mu}(q,t) := \lim_{m} \kappa_{\mu}^{(m)}(q,t) = \left( \sum_{i:\mu_i \neq 0}t^n \alpha_{\mu}^{(0)}(i) \right) + \frac{t^{1+\#\{j:\mu_j \neq 0\}}}{1-t} \in \mathbb{Q}(q,t).
\end{equation}

The last statement regarding $\kappa_{\mu*0^k}(q,t)$ and $\kappa_{s_i(\mu)}(q,t)$ follows now directly from Lemma \ref{convergence of eigenvalues} and classical DAHA intertwiner theory.
\end{proof}

\begin{remark*}
    Recall from the proof of Lemma \ref{convergence of eigenvalues} that 
    $$t^n\alpha_{\mu}^{(0)} = q^{\mu_i}t^{n+1-\beta_{\mu}(i)}.$$ Applying this to the Lemma \ref{kappas converge} gives the combinatorial formula 
    $$\kappa_{\mu}(q,t) = \frac{t^{1+\#\{j:\mu_j \neq 0\}}}{1-t}+ \sum_{\mu_i \neq 0} q^{\mu_i}t^{n+1- \beta_{\mu}(i)}.$$ If we consider the partition $\lambda$ to have an infinite string of $0$'s attached to its tail then 
    $$\kappa_{\lambda}(q,t) = \sum_{i=1}^{\infty} q^{\lambda_i}t^{i}.$$ Notice that this is exactly equal to 
    $$\frac{t}{1-t}\left(1 - (1-t)(1-q)B_{\lambda}(q,t) \right)$$ where $B_{\lambda}(q,t)$ is the diagram generator of $\lambda$ in \cite{HMacDG}.
\end{remark*}

\begin{cor}
    Let $\lambda$ and $\nu$ be partitions. Then $\kappa_{\lambda}(q,t) = \kappa_{\nu}(q,t)$ if and only if $\lambda = \nu$.
\end{cor}
\begin{proof}
    This follows readily from the identity 
    $$\kappa_{\lambda}(q,t) =  \sum_{i=1}^{\infty} q^{\lambda_i}t^{i}$$ given in the prior remark. 
\end{proof}

In this next result we will show that the sequence of operators $(\Psi_{p_1}^{(m)})_{m\geq 1}$ converges to a well defined map on $\sP_{as}^{+}$. As expected these operators are well-behaved on sequences of the form $\epsilon_{\ell(\mu)}^{(m)}(E_{\mu*\lambda*0^{m-(\ell(\mu)+\ell(\lambda))}}).$ In fact it is not hard to show that $(\Psi_{p_1}^{(m)})_{m\geq 1}$ converges on the former sequences. However, this is not a sufficient argument to show the convergence of the $(\Psi_{p_1}^{(m)})_{m\geq 1}$. In order to obtain a well-defined operator on $\sP_{as}^{+}$ from the sequence of operators $(\Psi_{p_1}^{(m)})_{m\geq 1}$ one needs to show that given an arbitrary convergent sequence $(f_m)_{m\geq 1}$ the corresponding sequence $(\Psi_{p_1}^{(m)}(f_m))_{m\geq 1}$ converges. Therefore, the difficulty in the following proof is to show that the $\Psi_{p_1}^{(m)}$ are well behaved in general.

\begin{thm}\label{psi 1 converges}
    The sequence of operators $(\Psi_{p_1}^{(m)})_{m\geq 1}$ converges to an operator $\Psi_{p_1}: \sP_{as}^{+}\rightarrow \sP_{as}^{+}$ which is diagonal in the $\widetilde{E}_{(\mu|\lambda)}$ basis with 
    $$\Psi_{p_1}(\widetilde{E}_{(\mu|\lambda)}) = \kappa_{\mu*\lambda}(q,t) \widetilde{E}_{(\mu|\lambda)}.$$
\end{thm}

\begin{proof}
    Notice that every element of $\sP_{as}^{+}$ is a finite $\mathbb{Q}(q,t)$-linear combination of terms of the form $T_{\sigma}x^{\lambda}F[X]$ where $\sigma$ is a permutation, $\lambda$ is a partition, and $F\in \Lambda$. Therefore, to show that the sequence of operators $(\Psi_{p_1}^{(m)})_{m\geq 1}$ converges it suffices using Remark \ref{useful defn of convergence} to show that sequences of the form 
    $$(\Psi_{p_1}^{(m)}(T_{\sigma}x^{\lambda}F[x_1+ \ldots +x_m]))_{m \geq 1} $$ converge. For $m$ sufficiently large, $T_{\sigma}$ commutes with $\Psi_{p_1}^{(m)} = t^m(Y_{1}^{(m)}+ \ldots + Y_{m}^{(m)})$ so it suffices to consider only sequences of the form $$(\Psi_{p_1}^{(m)}(x^{\lambda}F[x_1+ \ldots +x_m])) _{m \geq 1} .$$ Let $\lambda$ be a partition, $k := \ell(\lambda)$, $F\in \Lambda$, and take $m > k$. Recall that $\widetilde{Y}_1^{(m)}X_1 = t^mY_{1}^{(m)}X_1$ (\ref{defn of deformed y's}) from which it follows directly that $\widetilde{Y}_i^{(m)}X_i = t^mY_{i}^{(m)}X_i$ for all $1 \leq i \leq m$. Then for all $1\leq i \leq k$ we have that since $\lambda_i \neq 0$, 
    $$t^mY_i^{(m)}(x^{\lambda}F[x_1+\ldots+x_m]) = \widetilde{Y}_{i}^{(m)}(x^{\lambda}F[x_1+\ldots+x_m]).$$
    Therefore, 
    $$t^m(Y_1^{(m)}+ \ldots +Y_k^{(m)})(x^{\lambda}F[x_1+ \ldots +x_m]) = (\widetilde{Y}_{1}^{(m)}+ \ldots + \widetilde{Y}_{k}^{(m)})(x^{\lambda}F[x_1+ \ldots +x_m]).$$
    Now since $x^{\lambda}F[x_1+\ldots+x_m]$ is symmetric in the variables $\{k+1,\dots,m\}$ we see that 
    
    \begin{align*}
        &t^m(Y_{k+1}^{(m)}+ \ldots + Y_{m}^{(m)})(x^{\lambda}F[x_1+\ldots+x_m]) \\
        &= (t^{m-k}T_{k}\cdots T_1 \omega_m^{-1}T_{m-1}^{-1}\cdots T_{k+1}^{-1} + t^{m-k-1}T_{k+1}\cdots T_1 \omega^{-1}_{m}T_{m-1}^{-1}\cdots T_{k+2}^{-1} + \ldots+ tT_{m-1}\cdots T_1 \omega_{m}^{-1})(x^{\lambda}F[x_1+\ldots+x_m]) \\
        &= (t^{m-k}T_{k}\cdots T_1 + t^{m-k-1}T_{k+1}\cdots T_1 +\ldots+ tT_{m-1}\cdots T_1)\omega_m^{-1}(x^{\lambda}F[x_1+\ldots+x_m]) \\
        &= (t^{m-k}T_{k}\cdots T_1 + t^{m-k-1}T_{k+1}\cdots T_1 +\ldots+ tT_{m-1}\cdots T_1)(x_2^{\lambda_1}\cdots x_{k+1}^{\lambda_{k}}F[qx_1+x_2+ \ldots + x_m])\\
        &= (t^{m-k}+t^{m-k-1}T_{k+1}+\ldots+ tT_{m-1}\cdots T_{k+1}) \left( T_{k}\cdots T_1x_2^{\lambda_1}\cdots x_{k+1}^{\lambda_{k}}F[qx_1+x_2+ \ldots + x_m] \right).
    \end{align*}
Notice that since $T_{k}\cdots T_1x_2^{\lambda_1}\cdots x_{k+1}^{\lambda_{k}}F[qx_1+x_2+ \ldots + x_m]$ is symmetric in the variables $\{k+2,\ldots,m\}$ 
$$\epsilon_{k+1}^{(m)}(T_{k}\cdots T_1x_2^{\lambda_1}\cdots x_{k+1}^{\lambda_{k}}F[qx_1+x_2+ \ldots + x_m]) = T_{k}\cdots T_1x_2^{\lambda_1}\cdots x_{k+1}^{\lambda_{k}}F[qx_1+x_2+ \ldots + x_m].$$
Therefore, 
\begin{align*}
    &t^m(Y_{k+1}^{(m)}+ \ldots + Y_{m}^{(m)})(x^{\lambda}F[x_1+\ldots+x_m]) \\
    &= (t^{m-k}+\ldots+ tT_{m-1}\cdots T_{k+1})\epsilon_{k+1}^{(m)}(T_{k}\cdots T_1x_2^{\lambda_1}\cdots x_{k+1}^{\lambda_{k}}F[qx_1+x_2+ \ldots + x_m]) \\
    &= t(t^{m-k-1}+ \ldots + 1)\epsilon_{k}^{(m)}(T_{k}\cdots T_1x_2^{\lambda_1}\cdots x_{k+1}^{\lambda_{k}}F[qx_1+x_2+ \ldots + x_m])
\end{align*}
where the last equality follows from 
$$\left( \frac{t^{m-k-1}+ t^{m-k-2}T_{k+1}+\ldots+ T_{m-1}\cdots T_{k+1}}{t^{m-k-1}+ t^{m-k-2}+\ldots + 1} \right) \epsilon_{k+1}^{(m)} = \epsilon_{k}^{(m)}.$$ 
Putting it all together we see that 
\begin{align*}
    &\Psi_{p_1}^{(m)}(x^{\lambda}F[x_1+ \ldots +x_m])\\
    &= t^m(Y_1^{(m)}+\ldots+Y_m^{(m)})(x^{\lambda}F[x_1+ \ldots +x_m]) \\
    &= t^m(Y_1^{(m)}+\ldots+Y_k^{(m)})(x^{\lambda}F[x_1+ \ldots +x_m]) + t^m(Y_{k+1}^{(m)}+\ldots+Y_m^{(m)})(x^{\lambda}F[x_1+ \ldots +x_m]) \\
    &= (\widetilde{Y}_{1}^{(m)}+ \ldots + \widetilde{Y}_{k}^{(m)})(x^{\lambda}F[x_1+ \ldots +x_m]) + t(t^{m-k-1}+ \ldots + 1)\epsilon_{k}^{(m)}(T_{k}\cdots T_1x_2^{\lambda_1}\cdots x_{k+1}^{\lambda_{k}}F[qx_1+x_2+ \ldots + x_m])
\end{align*}

which by Theorem \ref{deformed Y's converge} and Corollary \ref{partial idempotents converge} converges to 
$$(\y_1+ \ldots + \y_k)(x^{\lambda}F[X]) + \frac{t}{1-t}\epsilon_{k}(T_k\cdots T_1x_2^{\lambda_1}\cdots x_{k+1}^{\lambda_k}F[X+(q-1)x_1]) .$$ Therefore, the limit operator $\Psi_{p_1}:= \lim_{m} \Psi_{p_1}^{(m)}$ is well defined. 

We will now show that the $\widetilde{E}_{(\mu|\lambda)}$ are weight vectors of $\Psi_{p_1}$ and compute their corresponding weight values. Let $\mu \in \Compred$ and $\lambda \in \Par$. By Corollary \ref{alternative description of E's} we have that 
$$ \widetilde{E}_{(\mu|\lambda)} = \lim_{m} \epsilon_{\ell(\mu)}^{(m)}(E_{\mu*\lambda*0^{m-(\ell(\mu)+\ell(\lambda))}}). $$
Therefore, by Proposition 6.21 from \cite{Ion_2022}, Lemma \ref{kappas converge}, and Lemma \ref{sym x and y commute with T's} it follows that 
\begin{align*}
   & \Psi_{p_1}(\widetilde{E}_{(\mu|\lambda)})\\
   &= \lim_{m} \Psi_{p_1}^{(m)}(\epsilon_{\ell(\mu)}^{(m)}(E_{\mu*\lambda*0^{m-(\ell(\mu)+\ell(\lambda))}}))\\
   &= \lim_{m} t^m(Y_1^{(m)}+\ldots+Y_m^{(m)})\epsilon_{\ell(\mu)}^{(m)}(E_{\mu*\lambda*0^{m-(\ell(\mu)+\ell(\lambda))}})\\
   &= \lim_{m} \epsilon_{\ell(\mu)}^{(m)}( t^m(Y_1^{(m)}+\ldots+Y_m^{(m)})E_{\mu*\lambda*0^{m-(\ell(\mu)+\ell(\lambda))}} )\\
   &= \lim_{m} \kappa_{\mu*\lambda}^{(m-(\ell(\mu)+\ell(\lambda)))}(q,t)\epsilon_{\ell(\mu)}^{(m)}(E_{\mu*\lambda*0^{m-(\ell(\mu)+\ell(\lambda))}})\\
   &= \kappa_{\mu*\lambda}(q,t)\widetilde{E}_{(\mu|\lambda)}. \\
\end{align*} 
\end{proof}

\begin{remark}\label{remark about delta'}
From the proof of Theorem \ref{psi 1 converges} we see that in particular, for partitions $\lambda$ we have that 
$$\Psi_{p_1}(\mathcal{A}_{\lambda}[X]) = \frac{t}{1-t}(1-(1-t)(1-q)B_{\lambda}(q,t))\mathcal{A}_{\lambda}[X] .$$ We saw that in Corollary \ref{alternative description of E's} $\mathcal{A}_{\lambda}[X] = (1-t)^{\ell(\lambda)}v_{\lambda}(t)P_{\lambda}[X;q^{-1},t]$ so,  following the argument of Haiman in \cite{HMacDG}, the operator $t^{-1}(1-t)\Psi_{p_1}$ is up to a change of variables equal to $\Delta'$. Therefore, we can view $t^{-1}(1-t)\Psi_{p_1}$ in a certain sense (after changing variables) as extending the operator $\Delta'$ from $\Lambda$ to $\sP_{as}^{+}$. Further, Theorem \ref{psi 1 converges} does not follow immediately from the work of Ion and Wu in \cite{Ion_2022} and in particular, 
$$ \Psi_{p_1} \neq \y_1+\y_2+\dots$$
although the latter operator is certainly well defined in a weak sense as a diagonal operator in the $\widetilde{E}_{(\mu|\lambda)}$ basis. The easiest way to see this is to note that $\y_1+\y_2+\dots$ will annihilate $\Lambda$ whereas $\Psi_{p_1}$ acting on the basis $\mathcal{A}_{\lambda}$ of $\Lambda$ has nonzero eigenvalues $\kappa_{\lambda}(q,t) \neq 0$.
\end{remark}

\begin{thm}[Second Main Theorem]\label{second main theorem}
    Let $\widetilde{Y}$ denote the $\mathbb{Q}(q,t)$-subalgebra of $End_{\mathbb{Q}(q,t)}(\sP_{as}^{+})$ generated by $\Psi_{p_1}$ and $\y_i$ for $i \geq 1$. $\sP_{as}^{+}$ has a basis of $\widetilde{Y}$-weight vectors and every $\widetilde{Y}$-weight space of $\sP_{as}^{+}$ is 1-dimensional.
\end{thm}
\begin{proof}
Since $\Psi_{p_1}$ is diagonal in the $\widetilde{E}_{(\mu|\lambda)}$ basis, see Theorem \ref{main theorem}, it commutes with each $\y_i$. Therefore, $\widetilde{Y}$ is a commutative algebra of operators on $\sP_{as}^{+}$ so it makes sense to ask about its weights in $\sP_{as}^{+}$. To show that the $\widetilde{Y}$-weight spaces of $\sP_{as}^{+}$ are 1-dimensional it suffices to show that if $(\mu^{(1)}|\lambda^{(1)}) \neq (\mu^{(2)}|\lambda^{(2)})$ for $\mu^{(1)},\mu^{(2)} \in \Compred$ and $\lambda^{(1)},\lambda^{(2)} \in \Par$ with $\widetilde{E}_{(\mu^{(1)}|\lambda^{(1)})}$ and $\widetilde{E}_{(\mu^{(2)}|\lambda^{(2)})}$ having the same $\y$-weight then the $\Psi_{p_1}$ eigenvalues for $\widetilde{E}_{(\mu^{(1)}|\lambda^{(1)})}$ and $\widetilde{E}_{(\mu^{(2)}|\lambda^{(2)})}$ are distinct. Necessarily, from the proof of Theorem \ref{main theorem}, if $\widetilde{E}_{(\mu^{(1)}|\lambda^{(1)})}$ and $\widetilde{E}_{(\mu^{(2)}|\lambda^{(2)})}$ have the same $\y$-weight then $\mu^{(1)} = \mu^{(2)} = \mu$. Since $(\mu|\lambda^{(1)}) \neq (\mu|\lambda^{(2)})$ it follows that $\lambda^{(1)} \neq \lambda^{(2)}$ so that $\sort(\mu*\lambda^{(1)}) \neq \sort(\mu*\lambda^{(2)} )$. From Lemma \ref{kappas converge} we then know that $\kappa_{\mu*\lambda^{(1)}} \neq \kappa_{\mu*\lambda^{(2)}}$ so lastly by Theorem \ref{psi 1 converges} we see that the $\Psi_{p_1}$ eigenvalues for $\widetilde{E}_{(\mu|\lambda^{(1)})}$ and $\widetilde{E}_{(\mu|\lambda^{(2)})}$ are distinct. Hence, the $\widetilde{Y}$-weight spaces of $\sP_{as}^{+}$ are 1-dimensional.

\end{proof}

Theorem \ref{psi 1 converges} motivates the following definition.

\begin{defn}
    For $F \in \Lambda$ let $\Psi_{F}: \sP_{as}^{+} \rightarrow \sP_{as}^{+}$ be the diagonal operator in $End_{\mathbb{Q}(q,t)}(\sP_{as}^{+})$ in the $\{\widetilde{E}_{(\mu|\lambda)}: \mu \in \Compred, \lambda \in \Par\}$ basis given by 
    $$\Psi_{F}(\widetilde{E}_{(\mu|\lambda)}) := F[\kappa_{\mu*\lambda}(q,t)]\widetilde{E}_{(\mu|\lambda)}.$$
\end{defn}

Notice that by construction every operator $\Psi_F$ commutes with the image of every $\y_i \in End_{\mathbb{Q}(q,t)}(\sP_{as}^{+})$ since from Corollary \ref{main theorem} we know that the $\widetilde{E}_{(\mu|\lambda)}$ are a basis of $\sP_{as}^{+}$.

\begin{conjecture}\label{conjecture}
For all $F \in \Lambda$ we have that 
$$\Psi_{F} = \lim_{m} \Psi_{F}^{(m)}.$$
\end{conjecture}

\begin{remark}\label{elliptic Hall algebra}
 Trivially, this conjecture holds for $F = 1 \in \Lambda$ and Theorem \ref{psi 1 converges} shows that this conjecture holds for $F = p_1$. Thus we see that the conjecture holds for $F \in \mathbb{Q}(q,t)[p_1]$. It is easy to extend part of the argument in the proof of Theorem \ref{psi 1 converges} to show that 
$$\lim_{m} \Psi_{F}^{(m)}(\epsilon_{\ell(\mu)}^{(m)}(E_{\mu*\lambda*0^{m-(\ell(\mu)+\ell(\lambda))}})) = F[\kappa_{\mu*\lambda}(q,t)] \widetilde{E}_{(\mu|\lambda)} = \Psi_{F}(\widetilde{E}_{(\mu|\lambda)}).$$
However, this is not sufficient to prove the conjecture. Similarly, to the proof of Theorem \ref{psi 1 converges} one needs to know that the sequence of operators $(\Psi_{F}^{(m)})_{m \geq 1}$ is well behaved on arbitrary convergent sequences in order to prove convergence to an operator in $ End_{\mathbb{Q}(q,t)}(\sP_{as}^{+})$. It would be sufficient to show that for every $r \geq 2$ the sequence of operators $(\Psi_{p_r}^{(m)})_{m\geq 1}$ converges since if this sequence converges its limit operator will necessarily agree with $\Psi_{p_r}$ on the $\widetilde{E}_{(\mu|\lambda)}$ basis.

This conjecture would imply the existence of an action of the elliptic Hall algebra \cite{BS} \cite{SV} on the space of almost symmetric functions. In independent work, according to private communications, Dongyu Wu has constructed an elliptic Hall algebra on $\sP_{as}^{+}$. Further, in Wu's action $P_{(0,1)}$ acts identically to $\Psi_{p_1}$ up to a scalar.
\end{remark}

\printbibliography

%%% MV:  Milo- your ref.bib won't let me edit it, but make sure you add your SLC article.  
%%% plus I don't yet see Burban-Schiffmann

\end{document}